\documentclass{amsart}
\usepackage[hmargin=3.5cm,vmargin=3.5cm]{geometry}
\usepackage[utf8]{inputenc}

\usepackage{amsmath
,amssymb
,amsthm
,stmaryrd
, mathtools, mathrsfs,
relsize
}
\usepackage{graphicx,svg,xcolor}
\usepackage{enumerate,blindtext}
\usepackage[linktocpage=true, backref = page]{hyperref}
\usepackage{todonotes, lipsum}
\usepackage[all,cmtip]{xy}
\usetikzlibrary{matrix}

\makeatletter
\g@addto@macro\bfseries{\boldmath}
\makeatother

\newtheorem{theorem}{Theorem}[section]
\newtheorem{lemma}[theorem]{Lemma}
\newtheorem{korollar}[theorem]{Corollary}

\newtheorem{proposition}[theorem]{Proposition}
\theoremstyle{definition}
\newtheorem{example}[theorem]{Example}
\newtheorem{remark}[theorem]{Remark}
\newtheorem{definition}[theorem]{Definition}

\newtheorem*{theorem*}{Theorem}

\newcommand{\G}{\ensuremath{\Gamma}}
\newcommand{\g}{\ensuremath{\gamma}}

\newcommand{\C}{\ensuremath{\mathbb{C}}}

\newcommand{\N}{\ensuremath{\mathbb{N}}}

\newcommand{\ov}{\ensuremath{\overline}}
\newcommand{\mf}{\ensuremath{\mathfrak}}
\newcommand{\mc}{\ensuremath{\mathcal}}

\newcommand{\R}{\ensuremath{\mathbb{R}}}
\newcommand{\inv}{\ensuremath{^{-1}}}

\renewcommand{\Re}{\ensuremath{\operatorname{Re}}}

\newcommand{\on}{\ensuremath{\operatorname}}

\DeclareMathOperator{\ad}{ad}

\DeclareMathOperator{\w}{w}

\DeclareMathOperator{\HC}{HC}

\makeatletter
\newcommand*\bigcdot{\mathpalette\bigcdot@{.5}}
\newcommand*\bigcdot@[2]{\mathbin{\vcenter{\hbox{\scalebox{#2}{$\m@th#1\bullet$}}}}}
\makeatother

\author{Tobias Weich}
\thanks{T. Weich: Paderborn University, Warburger Straße 100, 33100 Paderborn, Germany, e-mail: \href{mailto:weich@math.upb.de}{weich@math.upb.de}}

\author{Lasse L. Wolf} 
\thanks{L. L. Wolf (corresponding author): Paderborn University, Warburger Straße 100, 33100 Paderborn, Germany, e-mail: \href{mailto:llwolf@math.upb.de}{llwolf@math.upb.de}}

\title{Absence of principal eigenvalues for higher rank locally
symmetric spaces}
\begin{document}
 \begin{abstract}
Given a geometrically finite hyperbolic surface of infinite volume it is a classical result of Patterson that the
positive Laplace-Beltrami operator has no $L^2$-eigenvalues $\geq 1/4$. 
In this article we prove a generalization of this result for the 
joint $L^2$-eigenvalues of the algebra of commuting differential
operators on Riemannian locally symmetric spaces $\Gamma\backslash G/K$ of higher rank. We derive dynamical assumptions on the $\Gamma$-action on the geodesic and the Satake compactifications which imply the absence of the corresponding principal eigenvalues. A large class of examples fulfilling these assumptions are the non-compact quotients by Anosov subgroups. 
 \end{abstract}

\maketitle

%\tableofcontents
\section{Introduction}
Let $\mathbb H =SL(2,\R)/SO(2)$ be the hyperbolic plane equipped with the Riemannian metric of constant negative curvature and $\Gamma\subset SL(2,\R)$ a discrete torsion-free subgroup. 
Then $\Gamma\backslash \mathbb H$ is a Riemannian surface of constant negative curvature and the relations between the geometry of $\Gamma\backslash \mathbb H$, the group theoretic properties of $\Gamma$, the dynamical properties of the $\Gamma$-action on $\mathbb H$ or its compactification, and the spectrum of the positive Laplace-Beltrami operator $\Delta$ have been intensively studied over several decades. 
Let us focus on the discrete $L^2$-spectrum of the Laplace-Beltrami operator, i.e. those $\mu\in \R$ such that $(\Delta-\mu)f=0$ for some $f\in L^2(\Gamma\backslash \mathbb H)$, $f\neq 0$. If $\Gamma\subset SL(2,\R)$ is cocompact, then $\mu_0=0$ is always an eigenvalue corresponding to the constant function and Weyl's law for the elliptic selfadjoint operator $\Delta$ implies that there is a discrete set of infinitely many eigenvalues $0=\mu_0<\mu_1\leq\ldots$ of finite multiplicity. 
From a representation theoretic perspective there is a clear distinction between $\mu_i\in \,]0,1/4[$ and $\mu_i\geq 1/4$. The former correspond to complementary series representations and the latter to principal series representations occurring in $L^2(\Gamma\backslash SL(2,\mathbb R))$. We call the eigenvalues accordingly principal eigenvalues (if $\mu_i\geq1/4$) and complementary or exceptional eigenvalues (if $\mu_i\in ]0,1/4[$).
Merely by discreteness of the spectrum we know that there are at most finitely many complementary eigenvalues and infinitely many principal eigenvalues. 

If we pass to non-compact $\Gamma\backslash \mathbb H$, the situation becomes more intricate: For the modular surface $SL(2,\mathbb Z)\backslash \mathbb H$, which is non-compact but of finite volume, it is well known that there are no complementary eigenvalues but still infinitely many principal eigenvalues obeying a Weyl asymptotic. 
In general the question of existence of principal eigenvalues on finite volume hyperbolic surfaces is wide open. A long standing conjecture by Phillips and Sarnak \cite{Sarnak85}
states that for a generic lattice $\Gamma \subset SL(2,\mathbb R)$ there should be no principal eigenvalues.

If we pass to hyperbolic surfaces of infinite volume the situation is much better understood. A classical theorem by Patterson \cite{Patterson} states that if $\mathrm{vol}(\Gamma\backslash \mathbb H) = \infty$ and $\Gamma\subset SL(2,\mathbb R)$ is geometrically finite, then there are no principal eigenvalues. The result has later been generalized to real hyperbolic spaces of higher dimensions by Lax and Phillips \cite{laxphil}. Even if we are not aware of a reference, it seems folklore that the statement holds for general rank one locally symmetric spaces. 

In this article we are interested in a generalization of Patterson's theorem to higher rank locally symmetric spaces: 
                                                                                              
Let us briefly\footnote{A more detailed description of the setting will be provided in Section~\ref{sec:notation}.} introduce the setting: Let $X=G/K$ be a Riemannian symmetric space of non-compact type and 
$\Gamma\subset G$ a discrete torsion-free subgroup. We will 
be interested in the $L^2$-spectrum of the locally symmetric space 
$\Gamma\backslash X$. As for hyperbolic surfaces the Laplace-Beltrami operator is a canonical geometric differential operator whose spectral theory can be studied. If the symmetric space is of higher rank, there are however further $G$-invariant differential operators on $X$ that descend to differential operators on $\Gamma\backslash X$. 
It is from many perspectives more desirable to study the spectral theory of the whole algebra of invariant differential operators $\mathbb D(G/K)$ instead of just the spectrum of the Laplacian. In order to introduce the definition of the joint spectrum of $\mathbb D(G/K)$ we recall that $\mathbb D(G/K)$ is a commutative algebra generated by $r\geq 1$ algebraically independent differential operators and $r$ equals the rank of the symmetric space $X$. 
After a choice of generating differential operators a joint eigenvalue of these commuting differential operators would be given by an element in $\C^r$. 
A more intrinsic way of defining the spectrum which does not require to choose any generators, is provided by the Harish-Chandra isomorphism. 
This is an algebra isomorphism $\HC: \mathbb D(G/K) \to \text{Poly}(\mathfrak a^*)^W$ between the invariant differential operators and the complex-valued Weyl group invariant polynomials on the dual of $\mathfrak a =\rm{Lie}(A)$, where $A$ is the abelian subgroup of $G$  in the Iwasawa decomposition $G=KAN$. 
If we fix $\lambda \in \mathfrak a^*$ and compose the Harish-Chandra isomorphism with the evaluation of the polynomial at $\lambda$ we obtain a character $\chi_\lambda:= \text{ev}_\lambda\circ \HC : \mathbb D(G/K)\to \mathbb C$. With this notation we call $\lambda\in\mathfrak a_\C^*$ a joint $L^2$-eigenvalue on $\Gamma\backslash X$ if there exists 
$f\in L^2(\Gamma \backslash X)$ such that for all $D\in \mathbb D(G/K)$:
\[
 Df=\chi_\lambda(D) f.
\]
As for the hyperbolic surfaces we can distinguish two kinds of $L^2$-eigenvalues: The purely imaginary joint eigenvalues $\lambda \in i\mathfrak a^*$ correspond to principal series representations and we call them \emph{principal joint $L^2$-eigenvalues}. The remaining eigenvalues are called \emph{complementary} or \emph{exceptional} eigenvalues. These two kind of eigenvalues are not only distinguished by representation theory, but they also behave differently from the point of view of spectral theory: In their seminal paper \cite{DKV}, Duistermaat, Kolk and Varadarajan 
consider the case of cocompact discrete subgroups $\Gamma\subset G$. They prove that there exist infinitely many principal joint eigenvalues and their asymptotic growth is precisely described by a Weyl law with a remainder term. They furthermore prove an upper bound on the number of complementary eigenvalues whose growth rate is strictly inferior than the Weyl asymptotic of the principal eigenvalues.
There are thus much less complementary than principal eigenvalues.

The most prominent non-compact higher rank locally symmetric space is without doubt $\G\backslash X =SL(n,\mathbb Z) \backslash SL(n,\mathbb R)/ SO(n)$. By \cite{Mue07} (and in a more general setting by \cite{LindenVenkatesh}) it is known that there are infinitely many joint $L^2$-eigenvalues. Assuming the generalized Ramanujan conjecture which implies the absence of complementary eigenvalues (see e.g. \cite{Blomerbrumley}), we would get infinitely many principal joint $L^2$-eigenvalues. If one replaces the full modular group by a congruence subgroup $\G(n)$ of level $n\geq3$, the existence of infinitely many principal joint $L^2$-eigenvalues has been shown by Lapid and Müller \cite{lapidmueller09}. More precisely, there is a Weyl law for the principal joint eigenvalues and the number of complementary eigenvalues are shown to be bounded by a function of lower order growth. 

In the recent article \cite{edwardsoh22} Edwards and Oh give examples and conditions on the discrete subgroup $\G$ which imply that the complementary eigenvalues are not only of lower quantity but that they are indeed absent. The main example are selfjoinings of convex-cocompact subgroups in $PSO(n,1)$, but they conjecture that this holds for every Anosov subgroup. 
%Our main result below can also be applied to these examples (see Section~\ref{sec:examples}) in order to obtain precise knowledge about the whole joint eigenspectrum. #Bem: Wir haben ja noch gar nicht erklärt was unser Resultat ist. da wirkte der Satz auf mich deplaziert.

In this article we are interested in conditions on the group $\Gamma$ which imply the absence of principal eigenvalues. In order to state our main theorem, recall the definition of a wandering point: 
If $\Gamma$ acts continuously on a topological space $T$, then a point $t\in T$ is called wandering, if there exists a neighborhood $U\subset T$  of $t$ such that $\{\gamma \in \Gamma: \gamma U\cap U\neq \emptyset\}$ is finite. The collection of all wandering points is called the wandering set $\w(\Gamma,T)$. 

We can now state our main theorem.
\begin{theorem}\label{thm:main_intro}
	Let $X=G/K$ be a Riemannian symmetric space of non-compact type and 
$\Gamma\subset G$ a discrete torsion-free subgroup.
 Let $\ov X$ be the geodesic   or the maximal Satake compactification (see Sections~\ref{sec:geo} and \ref{sec:satake}) and let $\w(\G,\ov X)$ be the wandering set for the action of $\G$ on $\ov X$.
 If $\w(\G,\ov X)\cap\partial \ov X\neq \emptyset$, then there are no principal joint $L^2$-eigenvalues on $\Gamma\backslash X$.
\end{theorem}

Let us compare our theorem to the classical result of Patterson: 
First of all, for $\mathbb H$ the geodesic compactification and the Satake compactification coincide. Furthermore, if $\Gamma\subset SL(2,\R)$ is geometrically finite, then it is well known that the following are equivalent:
\begin{enumerate}
 \item $\mathrm{vol}(\Gamma\backslash \mathbb H) = \infty$
 \item the limit set of $\Gamma$ is not the whole boundary $\Lambda(\Gamma)\neq \partial \mathbb H$
 \item there is a non-empty open set of discontinuity $\Omega(\Gamma)\subset \partial \mathbb H$ on which $\Gamma$ acts properly discontinuously.
\end{enumerate}
The last point immediately implies the existence of a wandering point of the $\Gamma$ action on $\overline{ \mathbb H}$. In this sense our theorem boils down to the classical result of Patterson. Also the higher dimensional result of Lax-Phillips on $\mathbb H^n$ is easily recovered from our main theorem: If $\Gamma\subset PSO(1,n)$ is geometrically finite and $\Gamma\backslash \mathbb H^n$ of infinite volume, then at least one non-compact end has to be a funnel or a cusp of non-maximal rank, and the existence of such a non-compact end directly implies the wandering condition of Theorem~\ref{thm:main_intro}.

As discrete subgroups on higher rank semisimple Lie groups are known to be constrained by strong rigidity results, it is a valid question whether there are interesting examples in higher rank which fulfill the wandering condition of Theorem~\ref{thm:main_intro}. We address this question in Section~\ref{sec:examples} and we will see that   all images of Anosov representations fulfill our condition. This is a consequence of recent results on compactifications of Anosov symmetric spaces \cite{KL18,GGKW15} that are modeled on the Satake compactification. 

A further natural question is, whether one can also in the higher rank setting
obtain the result by the assumption of infinite volume of the locally symmetric spaces
instead of the dynamical assumption on the group action used in our theorem. We
do not know a definitive answer. However, it should be noted, that there is so
far no good notion of a geometrically finite group $\Gamma$ in higher rank.
Without the assumption of geometric finiteness, to our best knowledge even for
$SL(2,\mathbb R)$ it is unknown if infinite volume implies the absence of
principal eigenvalues. 
% By the analogy between a Jacobs ladder and a one dimensional chain of atoms
% as well as effects observed for Anderson localizations, we are tempted to
% conjecture that the answer is negative.

\emph{Outline of the proof and the article}.
Let $f\in C^ \infty(X)$ be the $\G$-invariant lift of a joint eigenfunction for $\mathbb D(X)$ that is in $L^ 2(\G\backslash X)$. 
The proof of Theorem~\ref{thm:main_intro} relies on the analysis of the asymptotic behavior of $f$ towards the boundary of the compactification at infinity. For the result on the geodesic compactification it suffices to study the asymptotics of $f$ into the regular directions. In order to obtain the result on the Satake compactification we are required to also analyze the behavior in singular directions along the different boundary strata of the Weyl chambers. 

In a first step we show that $f$ satisfies a certain growth condition called \emph{moderate growth}.
This is done by elliptic regularity combined with coarse estimates on the injectivity radius (see Section~\ref{sec:modgrowth}).

The knowledge of moderate growth then allows us (see Section~\ref{sec:absence}) to use asymptotic expansion results for $f$ by van den Ban-Schlichtkrull \cite{vdBanSchl87, vdBanSchl89LocalBD}. For the asymptotics into the regular directions, i.e. in the interior of the positive Weyl chamber $ \mf a^ +\subset \mf a$,  it follows from \cite{vdBanSchl87} that the leading term for the expansion of $f(k\exp(tH)K)$ with $k\in K$ and $H\in \mf a^ +$ is 
\[
\sum_{w\in W} p_w(k)e^ {(w\lambda-\rho)(tH)} \quad \text{as} \quad t\to \infty,
\]
where $W$ is the Weyl group, $\rho$ the usual half sum of roots and $\lambda\in i\mathfrak a^*$ a regular spectral parameter (for singular spectral parameters the formula becomes slightly more complicated but is still tractable). The wandering condition of $\G$ acting on the geodesic compactification  $X\cup X(\infty)$ yields a neighborhood $U$ in 
$X\cup X(\infty)$ of some point in $X(\infty)$ such that $f\in L^ 2(U\cap X)$.
Combining this with the expansion and the description of such neighborhoods $U$ implies that all  the boundary values $p_w$ vanish on an open subset of $K$. This implies, again by \cite{vdBanSchl87}, that $f=0$.

The result for the Satake compactification follows the same strategy but involves more complicated expansions from \cite{vdBanSchl89LocalBD} that describe the asymptotic behavior into the singular directions along the different boundary strata of the Weyl chamber (Section~\ref{sec:satabsence}). 

Finally, in Section~\ref{sec:examples} we provide some examples of higher rank locally symmetric spaces that fulfill the wandering condition of Theorem~\ref{thm:main_intro}. In particular, we show that all quotients by Anosov subgroups fulfill the assumption. 

% The difference between the geodesic and the maximal Satake compactification lies in the description of the neighborhoods so that we have to apply slightly different arguments.
% More precisely, for the geodesic compactification we only need the asymptotic expansion in the directions in the interior of the positive Weyl chamber $\mf a^+$. These are given in \cite{vdBanSchl87}.
% For the maximal Satake compactification we also need asymptotics along the walls of the positive Weyl chamber which are given in the later work \cite{vdBanSchl89LocalBD}.
% The proof for the geodesic compactification is given in Section~\ref{sec:geoabsence} and for the maximal Satake compactification in Section~\ref{sec:satabsence}.
% In Section~\ref{sec:conclusion} we prove Theorem~\ref{thm:main_intro}. After that we provide some examples including the case of Anosov subgroups (Section~\ref{sec:examples}).

\emph{Acknowledgement.}
We thank Valentin Blomer for his suggestion to study this question and for numerous stimulating discussions. We furthermore thank Samuel Edwards, Joachim Hilgert, Lizhen Ji, Fanny Kassel, Michael Magee, Werner Müller and Beatrice Pozzetti for discussions and advice to the literature.  This work has received funding from the Deutsche Forschungsgemeinschaft (DFG) Grant No. WE 6173/1-1 (Emmy Noether group “Microlocal Methods for Hyperbolic Dynamics”) as well as SFB-TRR 358/1 2023 — 491392403 (CRC ``Integral Structures in Geometry and Representation Theory'').

\section{Preliminaries}

\subsection{Symmetric spaces}\label{sec:notation}
In this section we fix the notation for the present article. Let $G$ be a real semisimple non-compact Lie group with finite center and  with Iwasawa decomposition $G=KAN$. Furthermore, let $M\coloneqq Z_K(A)$ be the centralizer of $A$ in $K$. 
%and $G=KAN_-$ the opposite Iwasawa decomposition. 
We denote by $\mf g, \mf a, \mf n, %\mf n_-,
\mf k,\mf m$ the corresponding Lie algebras. 
%For $g\in G$ let $H(g)$ be the logarithm of the $A$-component in the Iwasawa decomposition.   
We have a $K$-invariant inner product on $\mf g$ that is induced by the Killing form and the Cartan involution. 
We further have the orthogonal Bruhat decomposition $\mf g =\mf a \oplus \mf m \oplus \bigoplus _{\alpha\in\Sigma} \mf g_\alpha$ into root spaces $\mf g_\alpha$ with respect to the $\mf a$-action via the adjoint action $\ad$, i.e.
$\mf g_\alpha = \{Y\in \mf g\mid [H,Y]= \alpha(H) Y \;\forall H\in \mf a\}$.
Here $\Sigma=\{\alpha\in \mf a^ \ast\mid \mf g_\alpha\neq 0\}\subseteq \mf a^\ast$ is the set of restricted roots. 
Denote by $W$ the Weyl group of the root system of restricted roots. 
Let $n$ be the real rank of $G$ and $\Pi$ (resp. $\Sigma^+$) the simple (resp. positive) system in $\Sigma$ determined by the choice of the Iwasawa decomposition.  
Let $m_\alpha \coloneqq \dim_\R \mf g_\alpha$ and $\rho \coloneqq \frac 12 \Sigma_{\alpha\in \Sigma^+} m_\alpha \alpha$. 
%Denote by $w_0$ the longest Weyl group element, i.e. the unique element in $W$ mapping $\Pi$ to $-\Pi$. 
Let $\mf a_+ \coloneqq \{H\in \mf a\mid \alpha(H)>0 \,\forall \alpha\in\Pi\}$ denote the positive Weyl chamber. % and $\mf a^\ast_+$ the corresponding cone in $\mf a^\ast$ via the identification $\mf a \leftrightarrow \mf a^\ast$ through the Killing form $\langle\cdot,\cdot\rangle$ restricted to $\mf a$. 
%We denote by ${}_+\mf a^\ast$ the dual cone $\{\lambda \in \mf a ^\ast\mid \lambda (H)> 0 \,\forall H\in \ov{\mf a_+}\setminus\{0\}\}$ and by  $\ov{{}_+\mf a^\ast}$ its closure $\{\lambda \in \mf a ^\ast\mid \lambda (H)\geq 0 \,\forall H\in \mf a_+\}=\R_{\geq 0} \Pi$. 
%Hence, if $\omega_j$ is the dual basis of $\alpha_j$ then $\ov{{}_+\mf a^\ast}=\{\lambda\in \mf a^\ast\mid \langle \lambda,\omega_j\rangle\geq 0 \, \forall j=1,\ldots,n\}$. 
%Furthermore, we denote $\ov{ {}_-\mf a^\ast}\coloneqq -\ov{{}_+\mf a ^\ast}$. 
If $\ov {A^+} \coloneqq \exp (\ov {\mf a_+})$, then we have the Cartan decomposition $G=K\ov {A ^+}K$. 
The main object of our study is the symmetric space $X=G/K$ of non-compact type. On $X$ with a natural $G$-invariant measure $dx$ we have the integral formula 
\begin{align}\label{eq:intKAK}
 \int_{X} f(x)dx=\int_K\int_{\mf a_+} f(k\exp(H)) \prod_{\alpha\in\Sigma^+} \sinh(\alpha(H))^{m_\alpha} dHdk. 
\end{align}
(see \cite[Ch.~I Theorem~5.8]{gaga}).
 \begin{example}
  If $G=SL_n(\R)$, then we choose $K=SO(n)$, $A$ as the set of diagonal
matrices of positive entries with determinant 1, and $N$ as the set of upper
triangular matrices with 1's on the diagonal. $\mf a$ is the abelian Lie
algebra of diagonal matrices and the set of restricted roots is $\Sigma=
\{\varepsilon_i-\varepsilon_j\mid i\neq j\}$ where $\varepsilon_i(\lambda)$ is
the $i$-th diagonal entry of $\lambda$. The positive system corresponding to
the Iwasawa decomposition is  $\Sigma^+=\{\varepsilon_i-\varepsilon_j\mid
i<j\}$ with simple system $\Pi=\{\alpha_i = \varepsilon_i-\varepsilon_{i+1}\}$.
The positive Weyl chamber is $\mf
a_+=\{\operatorname{diag}(\lambda_1,\ldots,\lambda_n)\mid
\lambda_1>\cdots>\lambda_n\}$ and the Weyl group is the symmetric
group $S_n$ acting by permutation of the diagonal entries.  
\begin{figure}[ht]
  \centering
  \includegraphics[width=0.8\linewidth,trim = 5cm 15cm 9cm 10cm,clip]{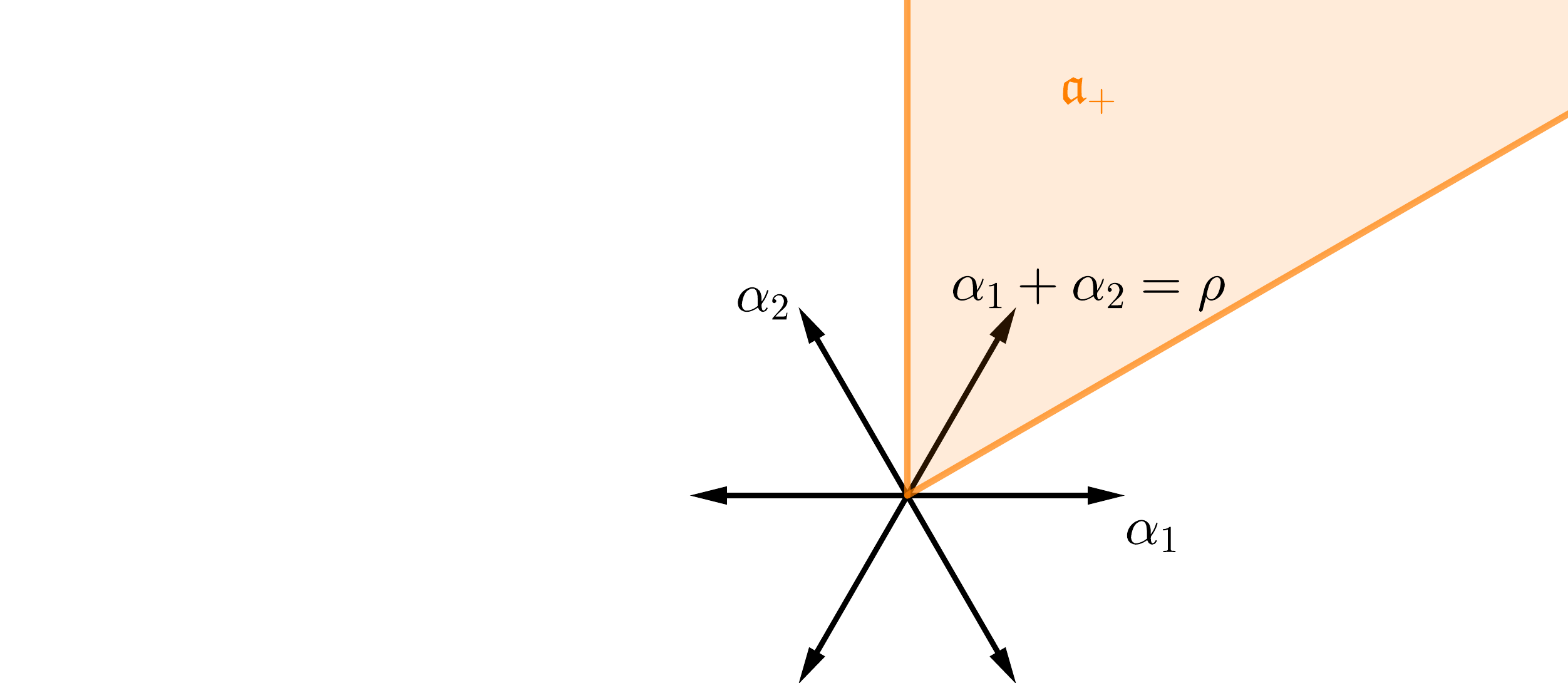}
  \caption{The root system for the special case $G=SL_3(\R)$: There are three positive roots $\Sigma^+=\{\alpha_1,\alpha_2,\alpha_1+\alpha_2\}$. As all root spaces are one dimensional the special element $\rho=\frac 12 \Sigma_{\alpha\in \Sigma^+} m_\alpha \alpha$ equals $\alpha_1+\alpha_2$.}
 \end{figure}
\end{example}

\subsection{Invariant differential operators}\label{sec:harishchandra}
Let $\mathbb D(G/K)$ be the algebra of \emph{$G$-invariant differential operators} on $G/K$, i.e. differential operators commuting with the left translation by elements  $g\in G$. Then we have an algebra isomorphism $\HC\colon \mathbb D(G/K)\to \text{Poly}(\mf a^\ast)^W$ from $\mathbb D(G/K)$ to the $W$-invariant complex polynomials on $\mf a^\ast$ which is called the \emph{Harish-Chandra homomorphism} (see \cite[Ch.~II~Theorem 5.18]{gaga}). For $\lambda\in \mf a^\ast_\C$ let $\chi_\lambda$ be the character of $\mathbb D(G/K)$ defined by $\chi_\lambda(D)\coloneqq \HC(D)(\lambda)$. Obviously, $\chi_\lambda= \chi_{w\lambda}$ for $w\in W$. Furthermore, the $\chi_\lambda$ exhaust all characters of $\mathbb D(G/K)$ (see \cite[Ch.~III Lemma~3.11]{gaga}). We define the space of joint eigenfunctions $$E_\lambda \coloneqq\{f\in C^\infty(G/K)\mid Df = \chi_\lambda (D) f \quad \forall D\in  \mathbb D(G/K)\}.$$ Note that $E_\lambda$ is $G$-invariant. 

\begin{example}
 	For $G=SL_n(\R)$ the algebra $\text{Poly}(\mf a_\C^\ast)^W$ is generated by $n-1$ elements 
	$p_2,\ldots,p_n$.
	Let us identify $\mf a_\C$ and $\mf a_\C^\ast$ via $\lambda\leftrightarrow \operatorname{Tr}(\lambda\; \cdot)$.
	Then $p_i(\lambda)= \lambda_1^i+\cdots+ \lambda_n^i=\operatorname{Tr}(\lambda^i)$ where $\lambda=\operatorname{diag}(\lambda_1,\ldots,\lambda_n)\in \mf a_\C$.
	Clearly, these polynomials are invariant under permutations of the diagonal entries and
	it can be shown that they are algebraically independent and generate $\text{Poly}(\mf a_\C^\ast)^W$ (see \cite{Hum92}).
	$\mathbb{D}(G/K)$ is then generated by the preimages of $p_i$ under $\HC$.
	Up to lower order terms the resulting invariant differential operators are given by the Maass-Selberg operators $\delta_i$ which are defined for $f\in
C^\infty(G/K)=C^\infty(SL_n(\R)/SO(n))$ 
by
\[
	\left.\delta_if(g K)=\operatorname{Tr}\left(\left(\frac{\partial}{\partial X}\right)^i\right) \right|_{X=0}f\left(g \exp\left(X-\frac{1}{n}\operatorname{Tr}(X)I_n \right)K\right),
\]
where 
\[
	X=\begin{pmatrix}
		x_{11}&\cdots&x_{1n}\\\vdots&\ddots&\vdots\\x_{1n}&\cdots&x_{nn}
	\end{pmatrix}
	\quad \text{and} \quad
\frac{\partial}{\partial X}= \begin{pmatrix}
	\frac{\partial}{\partial x_{11}}&\cdots&\frac{\partial}{2\partial x_{1n}}\\\vdots& \ddots &\vdots \\ \frac{\partial}{2\partial x_{1n}} &\cdots &\frac{\partial}{\partial x_{nn}}
\end{pmatrix}.
\]
(see \cite{brennecken2020algebraically}).
\end{example}
Now, let $\G\leq G$ be a torsion-free discrete subgroup. Since $D\in \mathbb D(G/K)$ is $G$-invariant, it descends to a differential operator ${}_\G D$ on the locally symmetric space $\G\backslash G/K$. Therefore, the left $\G$-invariant functions of $E_\lambda$ (denoted by ${}^\G E_\lambda$) can be identified with joint eigenfunctions on $\G\backslash G/K$ for each ${}_\G D$: $${}^\G E_\lambda = \{f\in C^\infty(\G\backslash G /K)\mid {}_\G D f = \chi_\lambda (D) f \quad \forall \, D\in \mathbb D(G/K)\}.$$

The goal is to show that $L^2(\G\backslash G/K)\cap {}^\G E_\lambda =\{0\}$ for $\lambda\in i\mf a^\ast$ and certain discrete subgroups $\G$. Then $$\sigma(\G\backslash X)\coloneqq\{\lambda\in\mf a^ \ast_\C\mid L^2(\G\backslash G/K)\cap {}^\G E_\lambda \neq\{0\}\}$$ has the property that the set of principal eigenvalues  $\sigma(\G\backslash X)\cap i\mf a^ \ast$ is empty.

\subsection{Geodesic compactification}\label{sec:geo}
In this section we recall the notion of the geodesic compactification of
a simply connected and non-positively curved Riemannian manifold $X$. A classical reference for this topic is \cite{Ebe96}. In the sequel also the Satake compactification will be crucial thus we provide detailed references to \cite{BJ} which treats both types of compactifications.
\begin{definition}[{\cite[Section I.2.2]{BJ}}]
 Two (unit speed) geodesics $\g_1,\g_2$ are equivalent if $\limsup_{t\to \infty} d(\g_1(t),\g_2(t))<\infty$. The space $X(\infty)$ is the factor space of all geodesics modulo this equivalence relation. The union $X\cup X(\infty)$ is called \emph{geodesic compactification}. The topology on $X\cup X(\infty)$ is given as follows: For $[\g]\in X(\infty)$ the intersection with $X$ of a fundamental system of neighborhoods is given by $C(\g, \varepsilon, R) = C(\g,\varepsilon)\smallsetminus B(R)$ where $$C(\g,\epsilon)=\{x\in X\mid\text{the angle between $\g$ and the geodesic from $x_0$ to $x$ is less than $\varepsilon$}\}$$ and $B(R)$ is the ball of radius $R$ centered at some base point $x_0\in X$. This topology is Hausdorff and compact.
\end{definition}

The space $X(\infty)$ can be canonically identified with the unit sphere in the tangent space at the base point $x_0 \in X$. If $\exp \colon T_{x_0} X\to X$ is the (Riemannian) exponential map at $x_0$, then a representative of the equivalence class of geodesics corresponding to a unit vector $Y\in T_{x_0} X$ is given by the geodesic $t\mapsto \exp(tY)$. %in particular, with the unit sphere in $\mf p =T_{x_0}X$ via $Y\mapsto \{t\mapsto \exp(tY)\}$ where $\mf p$ is the $-1$-eigenspace of the Cartan involution. 
This identification yields the neighborhoods $C(Y_0,\varepsilon,R) = \{\exp{tY}\mid t>R,\|Y\|=1,  |\cos\inv(\langle Y,Y_0\rangle)| < \varepsilon\}$ where $Y_0\in T_{x_0} X$ is normalized.
More precisely, if $\gamma$ is the geodesic $t\mapsto \exp(tY_0)$ then $C(\gamma,\varepsilon,R)=C(Y_0,\varepsilon,R)$.

%$\Omega_{\varepsilon,R}(Y_0)= \{\exp{tY}\mid t>R, \cos(\langle Y,Y_0\rangle) < \varepsilon\}$. 
                                                                                                
Let us return to the setting where $X=G/K$ is a symmetric space of non-compact type, then $X$ is simply connected and non-positively curved. Hence, the geodesic compactification of $X$ is defined and we have the following proposition.

\begin{proposition}[{\cite[Proposition I.2.5]{BJ}}]
 The action of $G$ on $X$ extends to a continuous action on $X\cup X(\infty)$.
\end{proposition}

% For $I\subsetneq \Pi$ let $\mf a_I=\bigcap_{\alpha\in I} \ker \alpha$, $\mf a^I=\mf a_I^\perp$, $\mf n_I=\bigoplus_{\alpha\in \Sigma^+\setminus\langle I\rangle} \mf g_\alpha$ and $\mf m_I = \mf m\oplus \mf a^I \oplus\bigoplus _{\alpha\in\langle I\rangle} \mf g_\alpha$. Define the subgroups $A_I=\exp \mf a_I$, $N_I=\exp \mf n_I$ and $M_I=M\langle\exp\mf m_I\rangle$. Then $P_I=M_IA_IN_I$ is the standard parabolic subgroup for the subset $I$.
% Let $\mf a_I^+ = \{H\in \mf a_I\mid \alpha (H)>0\,\forall \alpha\in \Pi\setminus I\}$. Then $\mf p = \bigcup_{k\in K, I\subsetneq \Pi} \Ad(k) \mf a_I^+$.
% \begin{proposition}
%  Let $Y\in S(\mf p)\cap \Ad(k)\mf a_I^+$ for $k\in K$ and $I\subsetneq \Pi$ where $S(\mf p)$ is the unit sphere in $\mf p$. Then the stabilizer of $Y$ is $kP_I k\inv$.
% \end{proposition}

\subsection{Maximal Satake compactification}\label{sec:satake}
In this section we introduce a different compactification for a Riemannian symmetric space $X=G/K$ the so called maximal Satake compactification. Before entering the technicalities let us give some heuristics: Recall that the Cartan decomposition allows to write $G=K\overline{\exp \mf a_+}K$ and since $K$ is compact the ``way'' in which a point in $G/K$ tends to infinity can be described in $\mathfrak a_+$. Recall that the particular simplicity of a rank one locally symmetric space stems from the fact that $\mathfrak a_+$ is just a half line (geometrically it corresponds to the distance from the origin of the symmetric space) and there is only one ``way'' to tend towards infinity. In the higher rank case $\mf a_+$ is a higher dimensional simplicial cone bounded by the hyperplanes $\ker \alpha\subset \mf a$ for $\alpha \in \Pi$ and the Satake compactifications will ``detect'' if a sequence tends to infinity inside the cone, while staying at bounded distance to a certain number of chamber walls $\ker \alpha $ for some subset $\alpha\in I\subsetneq \Pi$.

In order to describe the precise structure of the Satake compactification we need to introduce the following notion of standard parabolic subgroups:

For $I\subsetneq \Pi$ let $\mf a_I\coloneqq\bigcap_{\alpha\in I} \ker \alpha$, $\mf a^I\coloneqq \mf a_I^\perp$, $\mf n_I\coloneqq\bigoplus_{\alpha\in \Sigma^+\smallsetminus\langle I\rangle} \mf g_\alpha$ and $\mf m_I \coloneqq \mf m\oplus \mf a^I \oplus\bigoplus _{\alpha\in\langle I\rangle} \mf g_\alpha$. Define the subgroups $A_I\coloneqq\exp \mf a_I$, $N_I\coloneqq\exp \mf n_I$ and $M_I\coloneqq M\langle\exp\mf m_I\rangle$. Then $P_I\coloneqq M_IA_IN_I$ is the standard parabolic subgroup for the subset $I$.
We furthermore introduce the notation $\mf a^I_{+} \coloneqq \{H\in \mf a^I\mid \alpha (H)>0\;\forall \alpha\in I\}$ and $\mf a_{I,+} \coloneqq \{H\in \mf a_I\mid \alpha(H)>0\;\forall \alpha \in \Pi\smallsetminus I\}$.

\begin{example}
	For $G=SL_n(\R)$ the set of simple roots is $\Pi=\{\alpha_i=\varepsilon_i-\varepsilon_{i+1} \mid 1\leq i\leq n-1\}$.
	Let $I=\{\alpha_{i_1},\ldots,\alpha_{i_k}\}$ be a proper subset of $\Pi$.
	Then $\mf a_I= \{\on{diag}(\lambda_1,\ldots,\lambda_n)\mid \lambda_{i_j}=\lambda_{i_j+1}\}$
	and $\mf a^I= \bigoplus_j  \{\on{diag}(0,\ldots,\lambda_{i_j},-\lambda_{i_j+1},\ldots,0)\}$.
	Note that $\mf a^I=\on{span}\alpha_{i_j}$ if one identifies $\mf a$ and $\mf a^\ast$ (see Figure~\ref{fig:subspacesofa} for an illustration).
	Hence, $\mf a^I$ consists of blocks where a single block is a copy of the $\mf a$-part of $SL_m(\R)$.
	Each block corresponds to a root in $\Pi\smallsetminus I$.
	More precisely, if $\alpha_i\in \Pi\smallsetminus I$ then a block ends in row $i$. 
	Note that the $m_i$ can very well be equal to 1. 
	In this case there is simply a zero at this point on the diagonal.
	\begin{figure}[ht]
		\centering
		\includegraphics[width=0.8\textwidth, trim=5cm 6cm 4cm 4cm, clip]{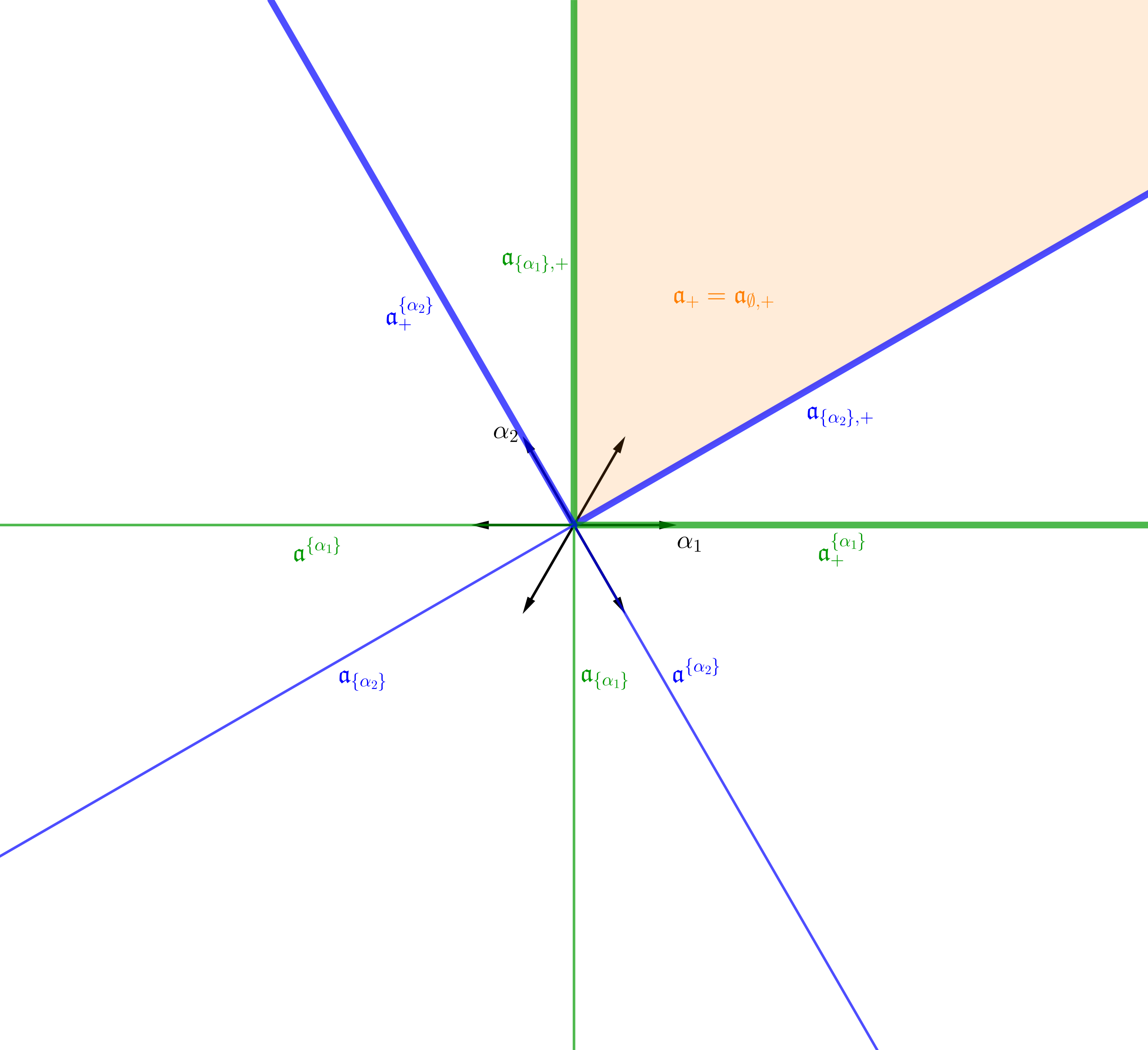}
		\caption{The various cones and subspaces in $\mf a$ corresponding to subsets of $\Pi$ for $G=SL_n(\R)$.
		$\mf a_\emptyset$ is all of $\mf a$ and $\mf a^\emptyset$ is the origin.}
		\label{fig:subspacesofa}
	\end{figure}
	$\mf m_I$ adds the corresponding root spaces, so $\mf m_I$ is isomorphic to direct sum of different $\mf {sl}_m(\R)$.
\[
	\mf m_I= \begin{pmatrix}
		\mf {sl}_{m_1}(\R)
	&&\\
	&
	\ddots
	&\\
	&&
	\mf {sl}_{m_{n-1-k}}(\R)
	\end{pmatrix}
\]	
where the bottom rows of the blocks correspond to the index of the roots in $\Pi\smallsetminus I$.
	$\mf n_I$ is the Lie algebra that contains of the upper-triangular matrices with non-zero entries in the positions that are not in the blocks of $\mf m_I$.
	On the group level $A_I= \{\on{diag}(\lambda_1,\ldots,\lambda_n)\in A\mid \lambda_{i_j}=\lambda_{i_j +1}\}$ and $N_I$ is the same as $\mf n_I$ but with $1$'s on the diagonal.
	For $M_I$ one has to multiply by $M=\{\on{diag}(\pm 1,\ldots,\pm 1)\}$ so that $M_I$ consists of block diagonal matrices where each block has determinant $\pm 1$ under the condition that the whole matrix has determinant $1$.
	It follows that the standard parabolic subgroups $P_I$ 
	are the sets of block upper-triangular matrices:
	\[
    \begin{tikzpicture}
\matrix (m1)    [matrix of nodes,
                 left delimiter={[}, right delimiter={]},
                 row sep=-0.5pt,column sep=-0.5pt,
                 every node/.style={inner sep=5pt}
                 ]
{
	|[draw]| $ \mbox{\Huge $\ast$}$
	 &            && $\mbox{\Huge $\ast$    }$          \\
	 & |[draw]|  $\ast$ &&               \\
             && $\ddots$ &               \\
	    \Huge 0 &&              & |[draw]| $\mbox{\Large $\ast$}$   \\
 };
    \end{tikzpicture}
		\]
		
\end{example}

The maximal Satake compactification $\ov X^{\max}$ is the $G$-compactification of $X$ (i.e. a compact Hausdorff space containing $X$ as an open dense subset such that the $G$-action extends continuously from $X$ to the compactification) with the orbit structure
$\ov X^{\max} = X \cup \bigcup_{I\subsetneq \Pi} \mc O_I$. 
For the orbit $\mc O_I$ we can choose a base point $x_I\in \mc O_I$ with $\operatorname{Stab}(x_I)=N_IA_I(M_I\cap K)$. 
The topology can be described as follows: Since $G = K\ov {A^+} K$ and $K$ is compact, it suffices to consider sequences $\exp H_n$, $H_n\in \ov {\mf a_+}$. 
Such a sequence by definition converges iff $\alpha(H_n)$ converges in $\R\cup\{\infty\}$ for all $\alpha\in \Pi$. 
If this is the case, to determine the limit, let $I=\{\alpha\in \Pi \mid \lim\alpha(H_n)<\infty \}$ and $H_\infty\in \ov{\mf a^I_+}$ such that $\alpha(H_\infty)=\lim\alpha(H_n)$ for $\alpha\in I$. Then $\exp H_n\to \exp (H_\infty)x_I$.

The intersection with $X$ of a fundamental system of neighborhoods of $k\exp(H_\infty)x_I$ with $k\in K, H_\infty\in {\mf a^I_+}$ is given by $$V\exp\{H\in \ov{\mf a_+}\mid |\alpha(H)-\alpha(H_\infty)|< \varepsilon,\alpha\in I, \alpha(H)>R,\alpha\not\in I\}x_0,$$ where $V$ is a fundamental system of neighborhoods of $k$ in $K$, $\varepsilon\searrow 0$, $R\nearrow \infty$.
If  $H_\infty\in \ov{\mf a^I_+}$, let $J=\{\alpha\in I\mid \alpha(H_\infty)=0\}$. Then the intersection with $X$ of a fundamental system of neighborhoods of $k\exp(H_\infty)x_I$ with $k\in K$ is given by $$V(K\cap M_J) \exp\{H\in \ov{\mf a_+}\mid |\alpha(H)-\alpha(H_\infty)|< \varepsilon, \alpha\in I, \alpha(H)>R,\alpha\not \in I\}x_0,$$ where $V,\varepsilon,R$ are as above.

\begin{figure}[ht]
	\centering
	\includegraphics[width=0.8\textwidth, trim=1cm 1cm 11cm 0cm, clip]{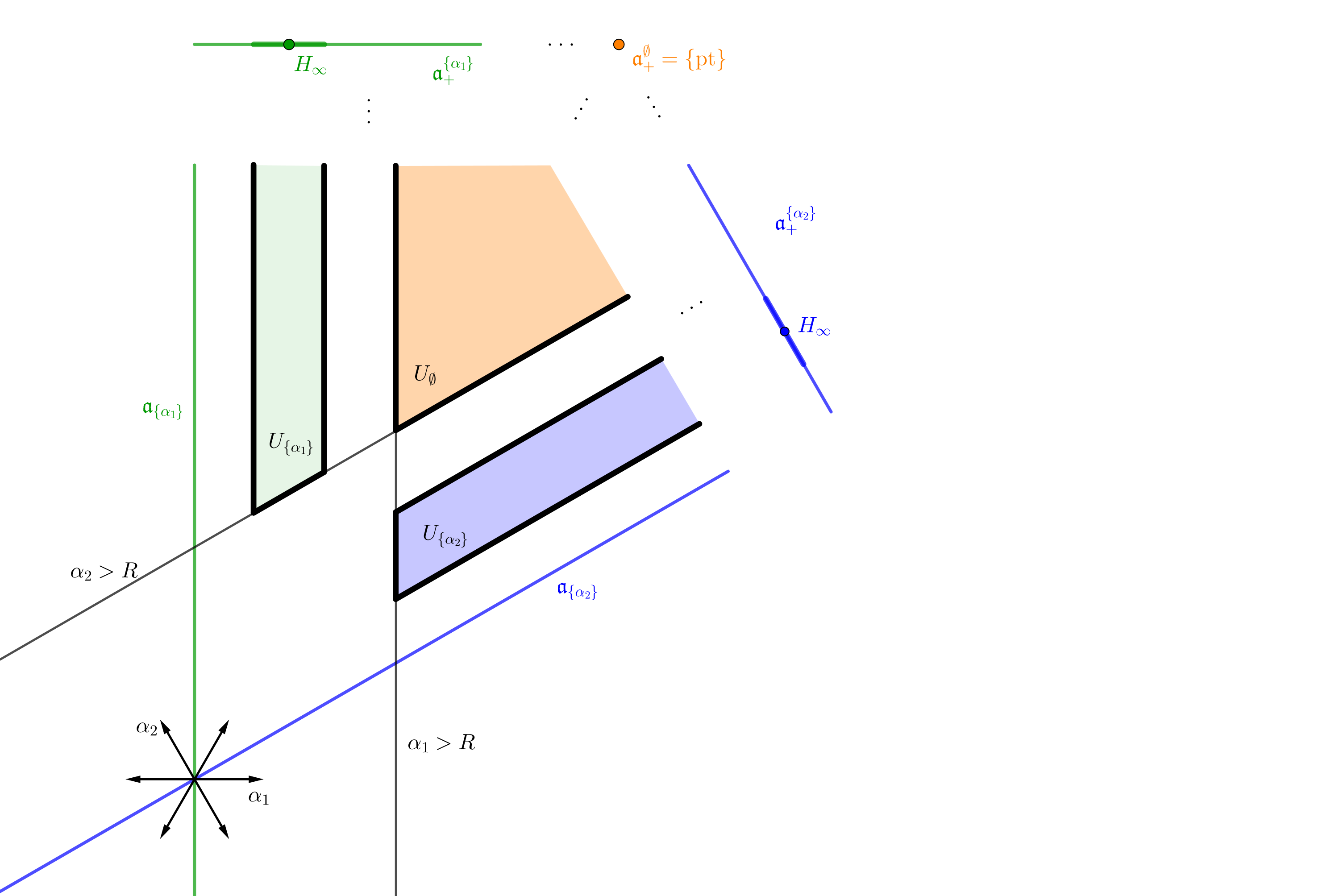}
	\caption{The compactification of $\mf a_+$ for $G=SL_3(\R)$ is obtained by gluing $\ov{\mf a_+^{\{\alpha_1\}}}, \ov{\mf a_+^{\{\alpha_2\}}}$ and $\ov{\mf a_+^\emptyset}$ to the boundary of $\ov{\mf a_+}$. 
	The sets $U_I$ for $I=\{\alpha_1\},\{\alpha_2\},\emptyset$ are the intersection of $\mf a_+$ with a fundamental neighborhood of $\exp(H_\infty)x_I$.}
	\label{fig:}
\end{figure}

Note that usually one defines the Satake compactification in a different way (see e.g. \cite[Ch.~I.4]{BJ}). Namely, let $\tau\colon G\to PSL(n,\C)$ be an irreducible faithful projective representation such that $\tau(K)\subseteq PSU(n)$. The closure in the projective space of Hermitian matrices of the image of the embedding of $X$ given by $gK\mapsto\R (\tau(g)\tau(g)^ \ast)$ is then called Satake compactification. It only depends on the highest weight $\chi_\tau$ of $\tau$. If $\chi_\tau$ is contained in the interior of the Weyl chamber, then this compactification is isomorphic to the maximal Satake compactification defined above. It is maximal in the sense that it dominates every other Satake compactification $\ov X^ S$ (i.e. there is a continuous $G$-equivariant map  $\ov X^{\max}\to \ov X^ S$). Since we only need the description of neighborhoods and the orbit structure we chose to introduce $\ov X^ {\max}$ this way.

\section{Moderate growth}\label{sec:modgrowth}
In this section we show that on a locally symmetric space each joint eigenfunction which is $L^2$ satisfies a growth condition in the following sense.
\begin{definition}
\begin{enumerate}[(i)]
 \item A function $f\colon X\to\C$ is called \emph{function of moderate growth} if there exist $r\in \R, C>0$ such that $$|f(x)|\leq C e^{rd(x,x_0)}$$ for all $x\in X$.
 \item For $\lambda\in \mf a_\C^ \ast$ the space $E_\lambda^\ast$ is the space of joint eigenfunction with moderate growth, i.e. 
 $$E_\lambda^\ast = \{f\in E_\lambda\mid f \text{ has moderate growth}\}.$$
\end{enumerate}

\end{definition}

Let $\G\leq G$ be a torsion-free discrete subgroup.%  and $d=\dim X$.

\begin{theorem}\label{thm:modgrowth}
 Let $f\in {}^\G E_\lambda\cap L^2(\G\backslash G/K)$. Then $f$ (considered as a $\Gamma$-invariant function on $X$) has moderate growth.
\end{theorem}

The proof uses Sobolev embedding and the following estimate on the injectivity radius.

\begin{proposition}[{see \cite[Thm. 4.3]{CGT}}]
 Let $(M,g)$ be a complete Riemannian manifold such that the sectional curvature $K_M$ satisfies $K_M\leq K$ for constant $K\in\R$.  Let $0< r <\pi/4\sqrt K$ if $K >0$ and $r\in (0,\infty)$ if $K\leq 0$.  Then  the  injectivity  radius $\mathrm{inj}(p)$ at $p$ satisfies $$\mathrm{inj}(p)\geq \frac{r\mathrm{Vol} (B_M(p,r))}{\mathrm{Vol} (B_M(p,r)) + \mathrm{Vol}_{T_pM}(B_{T_pM}(0,2r))},$$ where $\mathrm{Vol}_{T_pM}(B_{T_pM}(0,2r))$ denotes  the  volume  of  the  ball  of radius $2r$ in $T_pM$,  where  both  the volume  and  the  distance  function  are  defined  using  the  metric $g^\ast:= exp^\ast_p g$, i.e.  the  pull-back  of the metric $g$ to $T_pM$ via the exponential map.
\end{proposition}
For $M=\G\backslash G/K$ we obtain that the injectivity radius decreases at most exponentially.
\begin{proposition}\label{prop:injectivity}
 There are constants $C,s>0$ such that 
 $$\mathrm{inj}_{\G\backslash G /K}(\G x)\geq C\inv e^{-sd(x,eK)}$$ for every $x\in G/K$.
\end{proposition}
\begin{proof}
 Since $\G\backslash G/K$ is of non-positive curvature we can apply the above proposition for every $r> 0$. Note that $\exp\colon T_pM\to M$ is the universal cover of $M$ and therefore $\mathrm{Vol}_{T_{\G x}M}(B_{T_{\G x}M}(0,2r)) = \mathrm{Vol}_{G/K}(B_{G/K}(x,2r)) = \mathrm{Vol}_{G/K}(B_{G/K}(x_0,2r)) \leq C e^{sr}$ for some constants $C,s$ independent of $x$, where $x_0$ is the base point $eK$ of $G/K$.
 
 Hence, 
 $$\mathrm{inj}(\G x)\geq r (1+\mathrm{Vol}_{T_{\G x}M}(B_{T_{\G x}M}(0,2r))/\mathrm{Vol} (B_M(\G x,r)))\inv\geq r(1+ Ce^{sr}/\mathrm{Vol} (B_M(\G x,r)))\inv$$
 
 For $r=1+d(x,x_0)$ we have $B_M(\G x,r)\supseteq B_M(\G x_0,1)$ and therefore 
 $$\mathrm{inj}(\G x)\geq (1+d(x,x_0))  (1+ Ce^{s(1+d(x,x_0))}/\mathrm{Vol} (B_M(\G x_0,1))\inv \geq (1+C'e^{sd(x,x_0)})\inv.$$
 This finishes the proof.
\end{proof}

Note that this estimate isn't sharp.
Indeed, the growth rate $s$ that we obtain in the proof is independent of $\Gamma$ and only depends on the volume growth in $G/K$.

Let $m=\dim X$. We need the following well-known lemma on the geodesic balls in $G/K$.

\begin{lemma}\label{la:boxcounting}
 Fix $r>0$. There is a constant $C$ such that for every $x\in G/K$ and $\varepsilon > 0$ there is a finite set $A\subseteq B(x,r)$ such that $\bigcup_{a\in A} B(a,\varepsilon) \supseteq B(x,r)$ and $\#A\leq C \varepsilon^{-m}$.
\end{lemma}
\begin{proof}
 Let $a_1=x$ and choose inductively $a_{i+1} \in B(x,r)\smallsetminus \bigcup_{j=0}^i B(a_j,\varepsilon)$ if the latter is non-empty. This yields a finite set $A=\{a_1,\ldots,a_N\}$   (since $\overline {B(x,r)}$ is compact) such that $B(x,r+\varepsilon)\supseteq \bigcup_{j=0}^N B(a_j,\varepsilon)\supseteq B(x,r)$ and $B(a_j,\varepsilon/2)$ are pairwise disjoint. It follows that 
 $\mathrm{Vol}(B(x,r+\varepsilon))\geq \sum_i \mathrm{Vol}(B(a_i,\varepsilon/2) = \# A \cdot \mathrm{Vol}(B(x,\varepsilon/2))$ and therefore  $\# A\leq \frac C {\mathrm{Vol}(B(x,\varepsilon/2))}$. The lemma follows from the fact that the volume is independent from the center and decreases like $\varepsilon^m$ as $\varepsilon \to 0$. 
\end{proof}

We can now combine Proposition~\ref{prop:injectivity} with Sobolev embedding to prove Theorem~\ref{thm:modgrowth}.
\begin{proof}[Proof of Theorem~\ref{thm:modgrowth}]
 Since $B(x_0,1)$ is relatively compact, there exists a constant $C$ such that $$\sup_{x\in B(x_0,1)} |f(x)| \leq C \|(\Delta +1)^{m/4+\varepsilon}f\|_{L^2(B(x_0,1))} = C(\chi_\lambda(\Delta)+1)^{m/4+\varepsilon} \|f\|_{L^2(B(x_0,1))}$$ by ellipticity of the Laplace operator $\Delta$ on $G/K$ and the Sobolev embedding $H^{m/2+\varepsilon}(B(x_0,1))\hookrightarrow C(B(x_0,1))$. By $G$-invariance of $\Delta$ and $d$ the same holds true for $x_0$ replaced by an arbitrary point $x\in X$. In particular, 
 $$|f(x)|\leq C(\lambda) \|f\|_{L^2(B(x,1))}.$$
 By Proposition~\ref{prop:injectivity} there are constants $C,s>0$ independent of $x$ such that $\mathrm{inj}_{\G\backslash G/K}(\G y)\geq C\inv e^{-sd(x,eK)}$ for every $y\in B(x,1)$. Let $\varepsilon(x) \coloneqq \frac 1 C e^{-sd(x,eK)}$. Then there is a finite set $A(x)\subseteq B(x,1)$ such that $\bigcup_{a\in A(x)} B(a,\varepsilon(x))$ covers $ B(x,1)$ and $\#A(x)\leq C' \varepsilon(x)^{-m}$ by Lemma~\ref{la:boxcounting}. 
 Hence, 
 $$\|f\|_{L^2(B(x,1))}^2\leq \sum_{a\in A(x)} \|f \|_{L^2(B(a,\varepsilon(x)))}^2$$
 Since $\mathrm{inj}_{\G\backslash G/K}(\G a)\geq \varepsilon(x)$ we have $\|f \|_{L^2(B(a,\varepsilon(x)))}\leq \|f\|_{L^2(\G\backslash G/K)}$ for $a\in A(x)$.
 Therefore,
 \begin{align*}
   |f(x)|&\leq C(\lambda)  \|f\|_{L^2(\G\backslash G/K)}  \sqrt{\#A(x)}\\
   &\leq C(\lambda)  \|f\|_{L^2(\G\backslash G/K)} C'^{1/2} \varepsilon(x)^{-m/2}\\
   &= C(\lambda)  \|f\|_{L^2(\G\backslash G/K)}  C'^{1/2}C^{m/2} e^{msd(x,x_0)/2}. \qedhere
 \end{align*}
\end{proof}

\begin{remark}
 In the case of locally symmetric spaces of finite volume there is a different argument showing Theorem~\ref{thm:modgrowth}: If we lift $f$ to a function on $G$ which we also call $f$, then there is smooth compactly supported function $\alpha$ on $G$ such that $f=f\ast \alpha$ (see~\cite[Theorem~1]{HC66}). Then one easily shows that $|f(\G x)|\leq C\|f\|_{L^1(\G\backslash G/K)} e^{sd(x,x_0)}$ using simple estimates for lattice point counting. Since $L^2\subseteq L^1$ for spaces of finite volume, we can deduce moderate growth for $f$. Unfortunately, this argument does not work for infinite volume locally symmetric spaces since a pointwise bound including the $L^2$-norm of $f$ would need much better counting estimates.
\end{remark}

\section{Absence of imaginary values in the $L^2$-spectrum}\label{sec:absence}
We introduce the space of smooth vectors in $E_\lambda^\ast$. It is precisely the space of joint eigenfunctions with smooth boundary values (see \cite{vdBanSchl87}).
\begin{definition}
 $$E_\lambda^\infty = \{f\in E_\lambda\mid \exists\, r  \;\forall u\in\mc U(\mf g)\;\exists\, C_u>0\colon |(uf)(x)|\leq C_u e^{r d(x,x_0)}\}$$
\end{definition}

\subsection{Geodesic compactification}\label{sec:geoabsence}
In this section we want to prove the following theorem.
\begin{theorem}\label{thm:smoothsquarevanish}
	Let $f\in E_\lambda^\ast$, $\lambda\in i\mf a^\ast$, such that $f$ is square-integrable on $C(Y_0, \varepsilon,R)$ for some $\varepsilon, R, Y_0$ (see Section~\ref{sec:geo}). 
	Then $f=0$.
\end{theorem}

Let $X(\lambda)\coloneqq \{w\lambda-\rho-\mu\mid w\in W, \mu \in \N_0\Pi\}$ (see Figure~\ref{fig:X_lambda} for a visualization in example of $SL(3,\R)$).
 We will use the following asymptotic expansion for functions in $E_\lambda^\infty$.

\begin{theorem}[{\cite[Thm~3.5]{vdBanSchl87}}]\label{thm:expansion}
 For each $f\in E_\lambda^\infty$, $g\in G$, and $\xi \in X(\lambda)$ there is a unique polynomial $p_{\lambda,\xi}(f,x)$ on $\mf a$ which is smooth in $x$ such that 
 $$f(g\exp(tH))\sim\sum_{\xi\in X(\lambda)} p_{\lambda,\xi}(f,g,tH) e^{t\xi(H)},\quad  t\to \infty,$$ at every $H_0\in \mf a_+$, i.e. for every $N$ there exist a neighborhood $U$ of $H_0$ in $\mf a_+$, a neighborhood $V$ of $x$ in $G$, $\varepsilon>0$, $C>0$ such that 
 $$\left|f(y\exp(tH)) - \sum_{\Re \xi(H_0)\geq -N} p_{\lambda,\xi}(f,y,tH) e^{t\xi(H)}\right| \leq C e^{(-N-\varepsilon)t}$$ for all $y\in V, H\in U$, $t\geq 0$.
\end{theorem}
\begin{remark}
 The uniformity in $x$ is not stated in \cite{vdBanSchl87} but it follows from (6.18) therein.
\end{remark}

\begin{example}
	In the case where $G/K$ is the upper half plane $\mathbb{H}$ a simplified version of this theorem can be stated as follows.
	Suppose $f\in E_{s-1/2}^\infty$, i.e. $f\in C^\infty(\mathbb{H})$ with $\Delta f=s(1-s)f$ 
	and the derivatives of $f$ satisfy some uniform pointwise exponential bounds.
	We lift $f$ to a function (also called $f$) on the sphere bundle $S\mathbb{H}$ which is constant on the fibers.
	Denote by $\phi_t$ the geodesic flow.
	Then if $s\notin \frac 12 \mathbb{Z}$
	\[
		(\phi_t)_\ast f(x)\sim e^{-ts} \left( \sum_{n=0}^\infty p^+_n(x) e^{-nt} \right) +
		e^{-t(1-s)} \left(\sum_{n=0}^\infty p_n^- (x) e^{-nt}\right) 
	\]
with $p_n^\pm$ being smooth.
If $s\in \frac 12 \mathbb{Z}$ the functions $p_n^\pm$ can be polynomials of degree one in $t$.
\end{example}
\begin{figure}
\centering
\begin{minipage}{.5\textwidth}
  \centering
  \frame{\includegraphics[width=.8\linewidth, trim=11.5cm 1cm 10cm 6cm,clip]{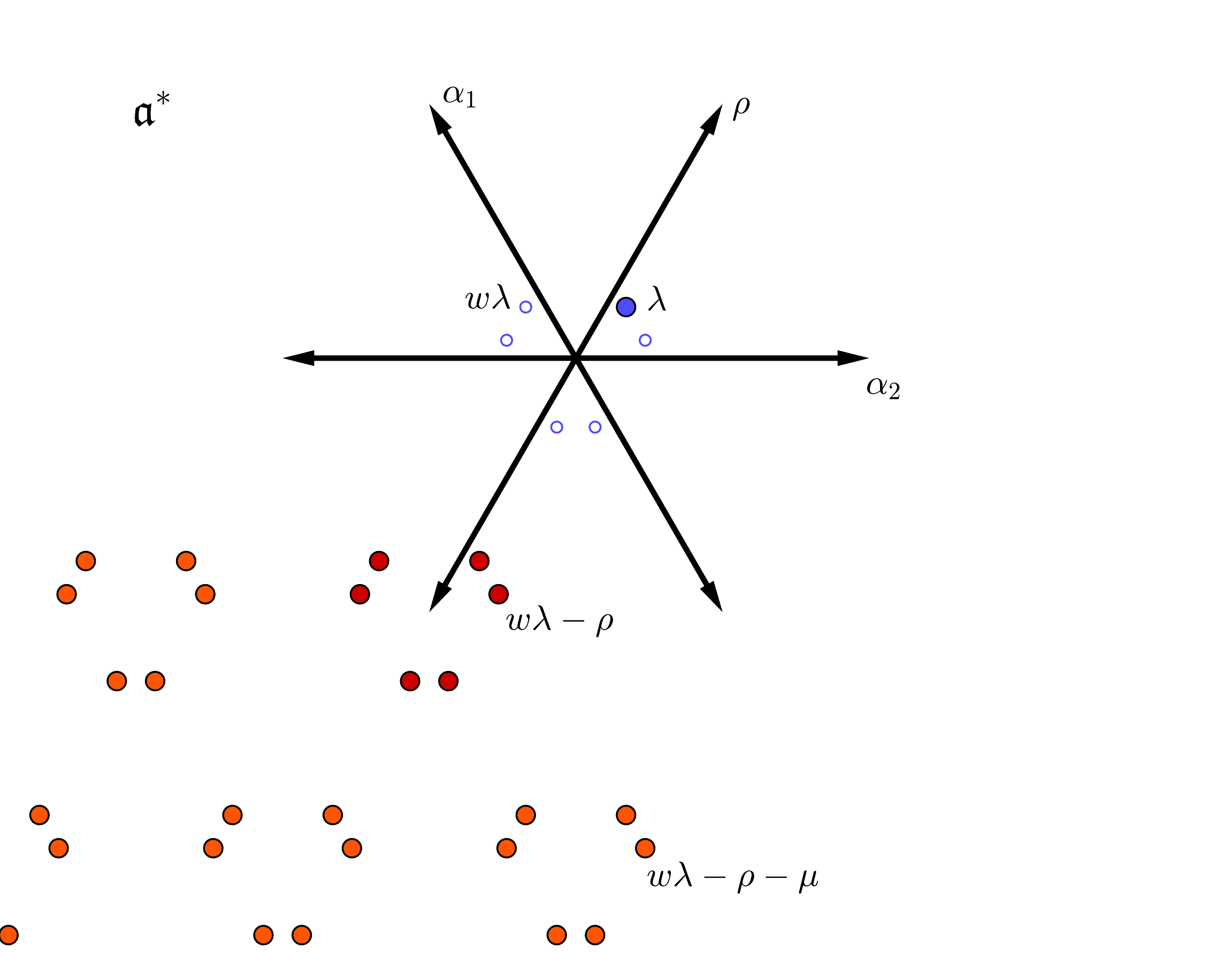}}%[trim=left bottom right top, clip]
\end{minipage}%
\begin{minipage}{.5\textwidth}
  \centering
  \frame{
  \includegraphics[width=.8\linewidth,trim=11.5cm 1cm 10cm 6cm,clip]{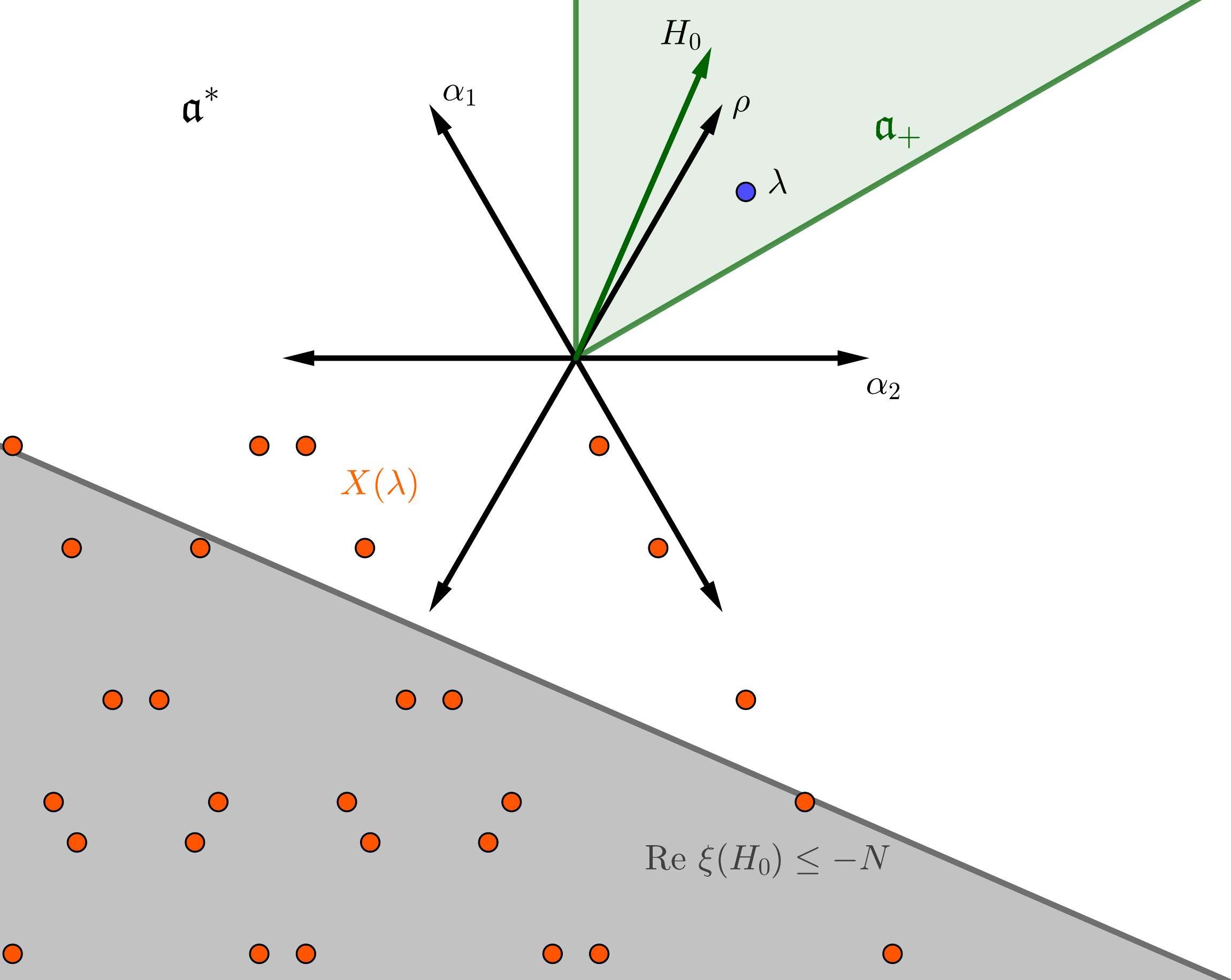}}
\end{minipage}
\caption{\label{fig:X_lambda} Real part of the exponents of the asymptotic expansion in Theorem~\ref{thm:expansion} for $G=SL(3,\R)$}
\label{fig:exponentsofexpansion}
\end{figure}

\begin{proof}[Proof of Theorem~\ref{thm:smoothsquarevanish}]
	First, we will consider the case $f\in E_\lambda^ \infty$.
By continuity there is a unit vector $H_0 \in \mf a_+$, a neighborhood $U$ of $H_0$ in the unit sphere of $\mf a$, and an open set $V$ in $K$ such that $$\Omega=\left\{k\exp(H)\colon k\in V, \frac{H}{\|H\|}\in U,\|H\|>R\right \}\subseteq C(Y_0,\varepsilon, R).$$
Let $N=\rho(H_0)$ such that without loss of generality
\begin{align}\label{eq:approximation}
 |f(k\exp(H)) - \sum_{w\in W} p_{\lambda,w\lambda-\rho}(f,k,H) e^{(w\lambda-\rho)(H)}| \leq C e^{(-\rho(H_0)-\varepsilon)\|H\|}
\end{align}
for all $k\in V, \frac H{\|H\|}\in U$.

We use the integral formula \eqref{eq:intKAK}
% \begin{align*}
%  \int_{G/K} f(x)dx=\int_K\int_{\mf a_+} f(k\exp(H)) \prod_{\alpha\in\Sigma^+} \sinh(\alpha(H))^{m_\alpha} dHdk. 
% \end{align*}
and observe that
\begin{align*}
 \int_{(R,\infty)U} &e^{-2(\rho(H_0)+\varepsilon)\|H\|}\prod_{\alpha\in\Sigma^+} \sinh(\alpha(H))^{m_\alpha} dH\\
 \leq& \int_{(R,\infty)U} e^{-2(\rho(H_0)+\varepsilon)\|H\|}e^{2\rho(H)}dH\leq  \int_{(R,\infty)U} e^{2(\rho(\frac H{\|H\|}-H_0)-\varepsilon)\|H\|}dH
\end{align*} which is finite after shrinking $U$ such that $\rho(\frac H{\|H\|}-H_0)<\varepsilon$ for $H\in U$. Consequently, the right hand side of \eqref{eq:approximation} and therefore also the left hand side of is square integrable on $\Omega$.

Since $f$ is $L^2$ and the approximation \eqref{eq:approximation} holds, $$\left|\sum_{w\in W} p_{\lambda,w\lambda-\rho}(f,k,H) e^{(w\lambda-\rho)(H)}\right|^2 \prod_{\alpha\in\Sigma^+} \sinh(\alpha(H))^{m_\alpha}$$ is integrable on $V\times (R,\infty)U$.
Hence, $\left|\sum_{w\in W} p_{\lambda,w\lambda-\rho}(f,k,H)
e^{(w\lambda-\rho)(H)}\right|^2 \prod_{\alpha\in\Sigma^+}
\sinh(\alpha(H))^{m_\alpha}$ is integrable on $(R,\infty)U$ for almost every
$k\in V$. Since $\sinh(x)\geq e^x/4$ for $x\geq \frac 12 \log 2$,
$$\left|\sum_{w\in W} p_{\lambda,w\lambda-\rho}(f,k,H)
e^{(w\lambda-\rho)(H)}\right|^2 e^{2\rho(H)}=\left|\sum_{w\in W}
p_{\lambda,w\lambda-\rho}(f,k,H) e^{w\lambda(H)}\right|^2$$ is integrable on
$(R,\infty)U$ for $R$ large enough. This is only possible if
$p_{\lambda,w\lambda-\rho}(f,k)$ vanishes on $\mf a$ for every $w\in W$ by
\cite[Lemma~8.50]{Kna86}. Since $p_{\lambda,w\lambda-\rho}(f,\bigcdot)$ is
smooth it vanishes identically on $V$.

We now show that it also vanishes on $VAN$.
For $n\in N$ \cite[Lemma~8.7]{vdBanSchl87} states for $f\in E_\lambda^\infty$
\[
	p_{\lambda,\xi}(f,n)=\sum_{\mu\in \N_0\Pi,\xi+\mu\in X(\lambda)}p_{\lambda,\xi+\mu}(f_\mu,e),\quad \xi\in X(\lambda),
\]
where $f_\mu\in L(\mc U(\mf g))f$ (where $L$ is the left regular representation) are specific joint eigenfunctions obtained by the Taylor expansion of $f$ in the direction of $n$ and $f_0=f$. 
For $\xi=w\lambda-\rho$ the only summand comes from $\mu=0$
since $\lambda\in i\mf a^\ast$ and $X(\lambda)=\{w\lambda-\rho-\mu \mid w\in W, \mu\in \N_0\Pi\}$.
In particular, $p_{\lambda,w\lambda-\rho}(f,n)=p_{\lambda,w\lambda-\rho}(f,e)$.

To deal with $a\in A$ we use \cite[Lemma~8.5]{vdBanSchl87}:
\[
	p_{\lambda,\xi}(f,a,H)=a^\xi p_{\lambda,\xi}(f,e,H+\log a),\quad f\in E_\lambda^\infty, \xi\in X(\lambda), H\in \mf a,
\]
where as usual $a^\xi=e^{\xi(\log a)}$.

Let us return to the situation that we achieved earlier, where $p_{\lambda,w\lambda-\rho}(f,k,H)=0$ for every $k\in V$ and $H\in \mf a$. 
But then 
\begin{align*}
	p_{\lambda,w\lambda-\rho}(f,kan,H)&=
	p_{\lambda,w\lambda-\rho}(L_{(ka)\inv}f,n,H)=
	p_{\lambda,w\lambda-\rho}(L_{(ka)\inv}f,e,H)\\
	&= 
	p_{\lambda,w\lambda-\rho}(L_{k\inv}f,a,H)=
	a^{w\lambda-\rho}p_{\lambda,w\lambda-\rho}(L_{k\inv}f,e,H)\\
	&=
	a^{w\lambda-\rho}p_{\lambda,w\lambda-\rho}(f,k,H)=0
\end{align*}
for every $k\in K, a\in A,n\in N$ and $w\in W$.
Hence, $p_{\lambda,w\lambda-\rho}(f,x)=0$ if $x$ is contained in the open set $VAN$.
This is exactly the assumption of \cite[Theorem~4.1]{vdBanSchl89LocalBD} in the case $I=I_\lambda$, i.e. $f$ is an eigenfunction for the whole algebra $\mathbb{D}(G/K)$ and is not only annihilated by an ideal of finite codimension.
Note that in this case $X(I)=X(\lambda)$.
We infer $f=0$.

It remains to show that the statement also holds for $f\in E_\lambda^\ast$.

% \begin{korollar}\label{cor:moderatesquarevanish}
% Let $f\in E_\lambda^\ast$, $\lambda\in i\mf a^\ast$, such that $f$ is square-integrable on $\Omega_{\varepsilon,R}(Y_0)$ for some $\varepsilon, R, Y_0$. Then $f=0$.
% \end{korollar}

 Since $C(Y_0,\varepsilon,R)$  is a fundamental system of neighborhoods  of $Y_0$ in the geodesic compactification and $G$ acts continuously on $X\cup X(\infty)$, there is a neighborhood $ V$ of $e$ in $G$ and $\varepsilon ', R'$ such that $V\inv C(Y_0,\varepsilon',R')\subseteq C(Y_0,\varepsilon,R)$. Let $\varphi_n$ be an approximate identity on $G$ with $\operatorname{supp} \varphi_n\subseteq V$, i.e. $\varphi_n \in C_c^\infty(G)$ is non-negative with $\int_G\varphi_n(g) dg =1$ and $\operatorname{supp}(\varphi_n)$ shrinks to $\{e\}$.  We consider $(\varphi_n\ast f)(x)=\int_G \varphi_n(g)f(g\inv x)dg$. Obviously, $\varphi_n\ast f\in E_\lambda^\infty$ since $L_x R_y(\varphi_n\ast f)=(L_x\varphi_n)\ast (R_yf), x,y\in G$.
 
Combining the already established case $f\in E_\lambda^\infty$ with
Lemma~\ref{la:conv_square_int} below we infer that $\varphi_n\ast f=0$ for all $n$
and therefore $f=0$. This completes the proof.  \end{proof}
\begin{lemma}\label{la:conv_square_int}
  $\varphi_n\ast f$ is square-integrable on $C(Y_0, \varepsilon',R')$.
 \end{lemma}
\begin{proof}[Proof of {Lemma~\ref{la:conv_square_int}}]
 Abbreviate $C'=C(Y_0, \varepsilon',R')$ and $C=C(Y_0, \varepsilon',R')$. It suffices to show that 
 \begin{align*}
  \left|\int_{C'} h(x)(\varphi_n\ast f)(x)dx\right|\leq B\|h\|_{L^2(C')}
 \end{align*}
for $h\in C_c(C')$ with a constant $B$ independent of $h$. 

Let us write $|h(x)\varphi_n(g)f(g\inv x)| =(|h|^2(x)\varphi_n(g))^{1/2} (|f|^2(g\inv x)\varphi_n(g))^{1/2}$ and use the Cauchy-Schwarz inequality of $L^2(V\times C')$ to obtain
\begin{align*}
 \left|\int_{C'} h(x)(\varphi_n\ast f)(x)dx\right|&\leq \int _V \int_{C'}|h(x)\varphi_n(g)f(g\inv x)|dx dg \\
 &\leq \left(\int_{V}\int_{C'} |h|^2(x) \varphi_n(g) dx dg \int_V \int_{C'} |f|^2(g\inv x) \varphi_n(g) dx dg\right)^{1/2}\\
 &\leq \|h\|_{L^2(C')} \left(\int_V \int_{C} |f|^2(x) \varphi_n(g) dx dg\right)^{1/2}\\
 &= \|h\|_{L^2(C')} \|f\|_{L^2(C)}
\end{align*}
where we used $V\inv C'\subseteq C$ in the last inequality. This finishes the proof.
\end{proof}

\subsection{Maximal Satake compactification}\label{sec:satabsence}
In this section we prove a statement analogous to Theorem~\ref{thm:smoothsquarevanish} for the maximal Satake compactification.
First of all we remark that each neighborhood of an element in the orbit $Gx_\emptyset\subseteq \ov X^ {\max}$ contains a neighborhood $C(Y_0,\varepsilon, R)$. Hence, we have the following proposition.
\begin{proposition}
 Let $f\in E_\lambda^\ast$, $\lambda\in i\mf a^\ast$, such that $f$ is square-integrable in some neighborhood of an element in $Gx_\emptyset\subseteq \ov X^ {\max}$. Then $f=0$.
\end{proposition}

The goal is to prove this statement for general neighborhoods in $\ov X^{\max}$. 
\begin{theorem}\label{thm:satakesquarevanish}
 Let $f\in E_\lambda^\ast$, $\lambda\in i\mf a^\ast$, such that $f$ is square-integrable in some neighborhood of an element $x_\infty \in \partial X^ {\max}$.
 %= k\exp(H_\infty)x_I\subseteq \ov X^ {\max}$ with $H_\infty\in \ov{\mf a^I_+}$. 
 Then $f=0$.
\end{theorem}

\begin{proof}
 By the same reasoning as in the proof of Theorem~\ref{thm:smoothsquarevanish} we can assume $f\in E_\lambda^\infty$. Moreover, we can assume that $x_\infty=k \exp(H_\infty)x_I$ with $k\in K$ and $H_\infty\in {\mf a^I_+}$ (instead of $H_\infty\in \ov{\mf a^I_+}$) since every neighborhood of $k\exp(H_\infty)x_I$ contains an element $k'\exp(H_\infty')x_I$ with $H_\infty'\in \mf a^I_+$. 
 
 Let $\Omega=V\exp (U)x_0\subseteq X$ with $U\coloneqq \{H\in \ov{\mf a_+}\mid |\alpha(H)-\alpha(H_\infty)|< \varepsilon,\alpha\in I, \alpha(H)>R,\alpha\not\in I\}$, so that $\Omega$ is the intersection of a neighborhood of $x_\infty$ with the interior of $\ov X^ {\max}$. Define $U^I\coloneqq \{H^I\in\mf a^I\mid |\alpha(H^I)-\alpha(H_\infty)|<\varepsilon, \alpha\in I\}$ which is a bounded open set in $\mf a^I$ since the set of linear forms $I$ restricted to $\mf a^I$ is linear independent. W.l.o.g. $U^I\subseteq \mf a^I_+$ has positive distance to the boundaries. Let $U_I\coloneqq \{H_I\in \mf a_I\mid \alpha(H_I)>C, \alpha\in \Pi\smallsetminus I\}\subseteq \mf a_{I,+}$ so that $U_I+U^I\subseteq U$ for $C$ large enough.
 
 \begin{figure}
  \includegraphics[width=.5\textwidth,trim=5cm 3cm 13cm 3cm, clip]{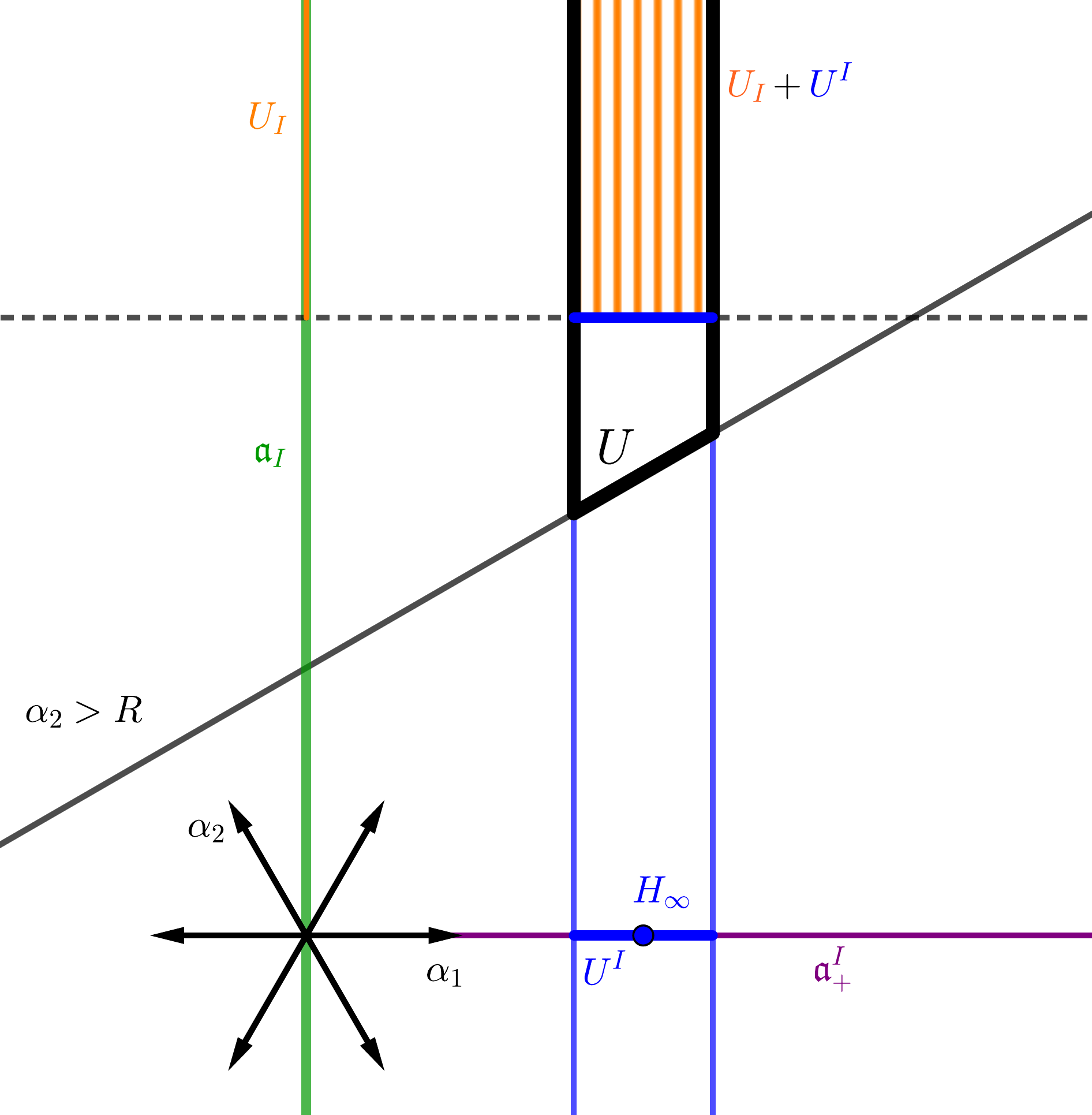}
  \caption{Decomposition of $U$ for $G=SL(3,\R)$ and $I=\{\alpha_1\}$.}
 \end{figure}

 As in Theorem~\ref{thm:smoothsquarevanish} we use the integral formula \eqref{eq:intKAK} % for $G=K\exp \ov{\mf a_+}K$ 
 to obtain 
 $$\int_{U\subseteq \mf a_+} |f|^2(k\exp(H))\prod \sinh(\alpha(H))^{m_\alpha}dH<\infty$$ for almost every $k\in V$.
 Therefore, 
 $$\int_{U_I\subseteq \mf a_{I,+}} |f|^2(k\exp(H^I)\exp(H_I))\prod \sinh(\alpha(H_I+H^I))^{m_\alpha}dH_I<\infty$$ for almost every $k\in V$ and $H^I\in U^I\subseteq \mf a^I$ (with suitable Lebesgue measures on $\mf a_I$ and $\mf a^I$). 
 
 The property that $U^I\subseteq \mf a^I_+$ has positive distance to the boundaries implies that $\alpha(H_I+H^I)>\varepsilon$ and hence $$\prod_{\alpha\in \Sigma^+} \sinh(\alpha(H_I+H^I))^{m_\alpha}\geq C e ^ {2\rho(H_I)},\quad H_I\in U_I,H^I\in U^I.$$
  Therefore, $|f|^2(k\exp(H^I)\exp(H_I))e^{2\rho(H_I)}$ is integrable on $U_I$.

  Similarly to the proof of Theorem~\ref{thm:smoothsquarevanish} we use an asymptotic expansion for $f$, but this time we have to consider asymptotics along the boundary of the positive Weyl chamber instead of regular directions.
  \begin{theorem}[{\cite[Thm~1.5]{vdBanSchl89LocalBD}}]
 There exists a finite set $S(\lambda,I)\subseteq \mf a_I^\ast$ such that for each $f\in E_\lambda^\infty$, $g\in G$, and $\xi \in X(\lambda,I)= S(\lambda,I) - \N_0\Pi|_{\mf a_I}$ there is a unique polynomial $p_{I,\xi}(f,x)$ on $\mf a_I$ which is smooth in $x$ such that 
 $$f(g\exp(tH_0))\sim\sum_{\xi\in X(\lambda,I)} p_{I,\xi}(f,g,tH_0) e^{t\xi(H_0)},\quad t\to\infty,$$ at every $H_0\in \mf a_{I,+}$, i.e. for every $N$ there exist a neighborhood $U$ of $H_0$ in $\mf a_{I,+}$, a neighborhood $V$ of $x$ in $G$, $\varepsilon>0$, $C>0$ such that 
 $$\left |f(y\exp(tH)) - \sum_{\Re \xi(H_0)\geq -N} p_{I,\xi}(f,y,tH) e^{t\xi(H)}\right| \leq C e^{(-N-\varepsilon)t}$$ for all $y\in V, H\in U$, $t\geq 0$.
\end{theorem}
\begin{remark}
 The uniformity in $x$ is not stated in \cite{vdBanSchl89LocalBD} but it follows from Proposition~1.3 therein.
\end{remark}
Let $H_0\in \mf a_{I,+}$, $\|H_0\|=1$. After shrinking we can assume that 
$$\left|f(k\exp(H^I)\exp(H_I)) - \sum_{\Re \xi(H_0)\geq -\rho(H_0)} p_{I,\xi}(f,k\exp(H^I),H_I) e^{\xi(H_I)}\right| \leq C e^{(-\rho(H_0)-\varepsilon)\|H_I\|}$$ for all $k\in V, H^I\in U^I$, and $\frac{H_I}{\|H_I\|}$ in some neighborhood $\tilde U_I$ of  $H_0$ in $\mf a_{I,+}$ such that $(R',\infty)\tilde U_I\subseteq U_I$.

The error term $e^{(-\rho(H_0)-\varepsilon)\|H_I\|}$ satisfies $$e^{2(-\rho(H_0)-\varepsilon)\|H_I\|}e^{2\rho(H_I)}=e^{2(\rho(-H_0+\frac{H_I}{\|H_I\|}) -\varepsilon)\|H_I\|}\leq e^{-\varepsilon\|H_I\|}$$
if $\tilde U_I$ is sufficiently small. Since $e^{-\varepsilon\|H_I\|}$ is integrable on $(R',\infty)\tilde U_I$ the same is true for $$\left |\sum_{\Re \xi(H_0)\geq -\rho(H_0)} p_{I,\xi}(f,k\exp(H^I),H_I) e^{(\xi+\rho)(H_I)}\right |^2 .$$
Using \cite[Lemma~8.50]{Kna86}  we obtain that  $p_{I,\xi}(f,k\exp(H^I),H_I) = 0$ if $\Re (\xi+\rho)(H_I) \geq 0$ for almost every $k\in V$ and $H^I\in U^I$. Since $p_{I,\xi}(f,\bigcdot,H_I)$ is smooth, this holds for every $k\in V$ and $H^I\in  U^I$.

By \cite[Corollary~2.5]{vdBanSchl89LocalBD} the mapping $M_I\ni m\mapsto p_{I,\xi}(f,xm,H_I)$, $x\in G$, is real analytic. Therefore, $\mf a^I\ni H^I\mapsto p_{I,\xi}(f,k\exp{H^I}, H_I)$ is real analytic as well and vanishes on the open set $U^I$ for $k\in V$, $\Re (\xi+\rho)(H_I)\geq 0$. Hence it vanishes on  $\mf a^I$ identically. 

In a last step of the proof we show that the vanishing of $p_{I,\xi}(f,k)$ for $\Re(\xi+\rho)(H_I)\geq 0$ implies that the expansion coefficients $p_{\lambda,\eta}(f,k)$ from Theorem~\ref{thm:expansion} vanish for all $\eta\in W\lambda-\rho$ and $k\in V$. For this purpose we use the following expansion for the polynomial $p_{I,\xi}$.
\begin{proposition}[{\cite[Theorem~3.1]{vdBanSchl89LocalBD}}]
 Let $f\in E_\lambda^\infty, g\in G, $ and $\xi\in X(I,\lambda)$.
 \begin{enumerate}
  \item For every $H_I\in \mf a_{I,+}$ and $H^I\in \mf a^I_+$ the following asymptotic expansion holds:
  $$p_{I,\xi}(f,g\exp(tH^I),H_I)\sim \sum _{\eta \in w\lambda -\rho - \N_0\Pi, \eta|_{\mf a_I}=\xi} p_{\lambda,\eta}(f,g,H_I+tH^I) e^{t\eta(H^I)}.$$
  \item For all $\eta = w\lambda -\rho - \N_0\Pi$ with $\eta|_{\mf a_I}\not\in X(\lambda,I)$ we have $p_{\lambda,\eta}(f,x)=0$.
 \end{enumerate}

\end{proposition}

Let $\eta=w\lambda-\rho$, $w\in W$, and $k\in V, H_I\in U_I$. If $\eta|_{\mf a_I}\not \in X(I,\lambda)$, then $p_{\lambda,\eta}(f,k)=0$. If $\eta|_{\mf a_I} = \xi \in X(I,\lambda)$, then $\Re(\xi+\rho)(H_I) = \Re w\lambda(H_I)=0\geq 0$. Therefore,
$p_{I,\xi}(f,k\exp{H^I}, H_I)= 0$ for all $H^I\in\mf a^I$ by the previous paragraph. It follows that the asymptotic expansion has every coefficient vanishing (see \cite[Lemma~3.2]{vdBanSchl87}), in particular  $ p_{\lambda,\eta}(f,k,H_I+tH^I)=0$, $H_I\in U_I$, $H^I\in \mf a^I$. Since $ p_{\lambda,\eta}(f,k)$ is a polynomial, this implies $p_{\lambda,\eta}(f,k)=0$.
Hence in both cases $p_{\lambda,w\lambda-\rho}(f,k)=0$ for $k\in V$.
The remainder of the proof proceeds the same way as the proof of Theorem~\ref{thm:smoothsquarevanish}.
\end{proof}

\subsection{Proof of Theorem~\ref{thm:main_intro}}\label{sec:conclusion}
Let $\ov X$ be one of the compactifications $X\cup X(\infty)$ or $\ov X^ {\max}$. 

Recall that the \emph{wandering set} $\w(\G,\ov X)$ is defined to be the points $x\in \ov X$  for which there is a neighborhood $U$ of $x$ such that $\g U\cap U\neq \emptyset$ for at most finitely many $\g\in \G$. Clearly, $\w(\G,\ov X)$ is open, $\G$-invariant and contains  $X$. Theorem~\ref{thm:main_intro} is a simple consequence of Theorem~\ref{thm:modgrowth} combined with Theorem~\ref{thm:smoothsquarevanish}, respectively \ref{thm:satakesquarevanish}.

\begin{proof}[Proof of Theorem~\ref{thm:main_intro}]
  Let $x\in \w(\G,\ov X) \cap \partial \ov X$. Hence, there is an open subset $U$ of $\ov X$ containing $x$ such that $\{\g\mid \g U\cap U\neq \emptyset\}$ contains $N\in \N$ elements. Let $\lambda\in i\mf a^\ast$ and $f\in L^2(\G\backslash X)$ a joint eigenfunction of $\mathbb D(X)$ for the character $\chi_\lambda$. Let $\ov f\in E_\lambda$ be $\G$-invariant lift of $f$ to $X$. Then 
  \begin{align*}
   \|\ov f\|^2_{L^2(U)} &= \int _U |\ov f|^2 = \int_{\G\backslash X} \sum_{\g\in \G} 1_U(\g y) |\ov f|^2(\g y) d (\G y) =  \int_{\G\backslash X} \# \{\g\mid \g y\in U\} | f|^2(\G y) d (\G y)\\
   &\leq N \|f\|^2_{L^2(\G\backslash X)}<\infty.
  \end{align*}
 Hence, $f$ is $L^2$ on $U$ and $f$ is of moderate growth by Theorem~\ref{thm:modgrowth}. Using Theorem~\ref{thm:smoothsquarevanish} or \ref{thm:satakesquarevanish} finishes the proof.
\end{proof}

\section{Examples}\label{sec:examples}
In this section we discuss three classes of examples that satisfy the wandering condition of Theorem~\ref{thm:main_intro}. 
As mentioned in the introduction the condition is satisfied for geometrically finite discrete subgroups of $PSO(n,1)$ of infinite covolume.

\subsection*{Products}
The most basic example is the case of products. Let $X=X_1\times X_2$ be the product of two symmetric spaces of non-compact type where $X_i = G_i/K_i$. 
Let $\Gamma\leq G_1\times G_2$ be a discrete torsion-free subgroup that is the product of two discrete torsion-free subgroups $\G_i\leq G_i$.
Then it is clear that the spectral theory of $\G\backslash X$ is completely determined by the spectral theory of the two factors.
In particular, since the algebra $\mathbb D(G/K)$ is generated by $\mathbb D(G_i/K_i)$, $i=1,2$, there are no principal joint eigenvalues if the same holds for one of the factors.
The same statement can be obtained by Theorem~\ref{thm:main_intro} using the maximal Satake compactification.
Indeed, by \cite[Prop. I.4.35]{BJ} it holds that the maximal Satake compactification of $X$ is the product of the maximal Satake compactifications of $X_i$, i.e. $\ov X ^{\max}=  \ov {X_1} ^{\max} \times \ov {X_2} ^{\max}$.
Then it is clear from the definition of the wandering set that $\w(\G,\ov X^{\max})=\w(\G_1,\ov {X_1} ^{\max})\times \w(\G_2,\ov {X_2} ^{\max})$. 
Hence, the wandering condition $\w(\G,\ov X ^{\max} )\cap \partial \ov X^ {\max} \neq \emptyset$ is fulfilled if and only if it is fulfilled for one of the actions $\G_i\curvearrowright \ov {X_i}^{\max}$. 

\subsection*{Selfjoinings}
A more interesting class of examples is given by selfjoinings of locally symmetric spaces.  
These are given as follows. As above let $X=X_1\times X_2$ be the product of two symmetric spaces of non-compact type where $X_i = G_i/K_i$. Now, let $\Upsilon$ be a discrete group and $\rho_i\colon \Upsilon\to G_i$, $i=1,2$, two representations into real semisimple non-compact Lie groups with finite center. 
We assume that $\rho_1$  has finite kernel and discrete image.
We want to consider the subgroup $\G$ of $G_1\times G_2$ given by $\G=\{(\rho_1(\sigma),\rho_2(\sigma))\colon \sigma\in \Upsilon\}$ which is discrete. 
We assume that $ \G$ is  torsion-free (e.g.  if $\Upsilon$ is torsion-free). 
In contrast to the previous example the locally symmetric space $\G\backslash X$ is not a product of two locally symmetric spaces anymore, so also the spectral theory cannot be reduced to some lower rank factors. 
However, we can exploit that the globally symmetric space is still a product and consider the maximal Satake compactification which is given by $\ov X ^{\max}=  \ov {X_1} ^{\max} \times \ov {X_2} ^{\max}$.
Since $\rho_1(\Upsilon)$ is discrete, it acts properly discontinuously on $X_1$.
Hence every point of $X_1$ is wandering for the action of $\rho_1(\Upsilon)$.
It follows easily that $X_1\times  \ov {X_2} ^{\max}$ is contained in the wandering set $\w(\G,\ov X^{\max})$ of the action $\G\curvearrowright \ov X^ {\max}$.
Therefore, the wandering condition is fulfilled. Indeed, $\w(\G,\ov X^{\max}) \cap \partial\ov X^{\max} \supseteq X_1\times \partial \ov {X_2} ^{\max}  \neq \emptyset$.

\subsection*{Anosov subgroups}
The result of Lax and Phillips \cite{laxphil} is in particular true if we consider a (non-cocompact) convex-cocompact subgroup of $PSO(n,1)$. 
Anosov subgroups as introduced by Labourie \cite{Lab06}  for surface groups and generalized to arbitrary word hyperbolic groups by Guichard and Wienhard \cite{GW12} generalize convex-cocompact subgroups to higher rank symmetric spaces. 
For such $\G$ we have the following proposition.
\begin{proposition}
 Let $\G$ be a torsion-free Anosov subgroup that is not a cocompact lattice in a rank one Lie group.
 Then the wandering condition $\w(\G, \ov X^ {\max})\cap \partial \ov X^ {\max} \neq \emptyset$ is fulfilled.
\end{proposition}
\begin{proof}
 By \cite{KL18} (and \cite{GGKW15} for a specific maximal parabolic subgroup) every locally symmetric space arising from an  Anosov  subgroup admits a compactification modeled on the maximal Satake compactification $\ov X^{\max}$, i.e. there is $X\subseteq \Omega \subseteq \ov X^ {\max}$ open such that $\G$ acts properly discontinuously and cocompactly on $\Omega$.
Since $\Gamma$ does not act cocompactly on $X$, we have $ \Omega \cap \partial \ov X^ {\max} \neq \emptyset$.
As every point in a region of discontinuity is wandering by definition we have $\Omega \subseteq \w(\G,\ov X^ {\max})$, and in particular the wandering condition is fulfilled.
\end{proof}

Combining the above proposition with Theorem~\ref{thm:main_intro} we obtain the following corollary.
\begin{korollar}
Let $\G$ be a torsion-free Anosov subgroup that is not a cocompact lattice in a rank one Lie group.
  Then there are no principal joint $L^2$-eigenvalues on $\G\backslash X$.
\end{korollar}

It is worth mentioning that selfjoinings of two representations into $PSO(n,1)$ yield Anosov subgroups if and only if one of the images of the representations is convex-cocompact. One can thus easily construct non-trivial examples which are not Anosov subgroups but fulfill the wandering condition of Theorem~\ref{thm:main_intro}. This is again parallel to Patterson's result that holds beyond the convex-cocompact case for hyperbolic surfaces admitting cusps and at least one funnel.

\section*{Declaration}
This work has received funding from the Deutsche Forschungsgemeinschaft (DFG) via Grant No. WE 6173/1-1 (Emmy Noether group “Microlocal Methods for Hyperbolic Dynamics”) as well as SFB-TRR 358/1 2023 — 491392403 (CRC ``Integral Structures in Geometry and Representation Theory'').

    The authors have no relevant financial or non-financial interests to disclose.

    The authors have no competing interests to declare that are relevant to the content of this article.

    All authors certify that they have no affiliations with or involvement in any organization or entity with any financial interest or non-financial interest in the subject matter or materials discussed in this manuscript.

    The authors have no financial or proprietary interests in any material discussed in this article.

%\newpage
\bibliographystyle{amsalpha}
\bibliography{literatur}

\providecommand{\bysame}{\leavevmode\hbox to3em{\hrulefill}\thinspace}
\providecommand{\MR}{\relax\ifhmode\unskip\space\fi MR }
\providecommand{\MRhref}[2]{%
  \href{http://www.ams.org/mathscinet-getitem?mr=#1}{#2}
}
\providecommand{\href}[2]{#2}
\begin{thebibliography}{GKW15}

\bibitem[BB13]{Blomerbrumley}
V.~Blomer and F.~Brumley, \emph{The role of the {R}amanujan conjecture in
  analytic number theory}, Bull. Amer. Math. Soc. \textbf{50} (2013), no.~2,
  267--320.

\bibitem[BCH20]{brennecken2020algebraically}
D.~Brennecken, L.~Ciardo, and J.~Hilgert, \emph{{Algebraically Independent
  Generators for the Algebra of Invariant Differential Operators on
  $\mathrm{SL}_n(\mathbb R)/\mathrm{SO}_n(\mathbb R)$}}, arXiv:2008.07479
  (2020).

\bibitem[BJ06]{BJ}
A.~Borel and L.~Ji, \emph{Compactifications of symmetric and locally symmetric
  spaces}, Mathematics: Theory and Applications, Birkhäuser Boston, 2006.

\bibitem[CGT82]{CGT}
J.~Cheeger, M.~Gromov, and M.~Taylor, \emph{Finite propagation speed, kernel
  estimates for functions of the {L}aplace operator, and the geometry of
  complete {R}iemannian manifolds}, J. Differential Geom. \textbf{17} (1982),
  no.~1, 15--53.

\bibitem[DKV79]{DKV}
J.J. Duistermaat, J.~Kolk, and V.~Varadarajan, \emph{Spectra of compact locally
  symmetric manifolds of negative curvature}, Invent. Math. \textbf{52} (1979),
  27--93.

\bibitem[Ebe96]{Ebe96}
P.~Eberlein, \emph{Geometry of nonpositively curved manifolds}, University of
  Chicago Press, 1996.

\bibitem[EO22]{edwardsoh22}
S.~Edwards and H.~Oh, \emph{Temperedness of {$L^2(\Gamma\backslash G)$} and
  positive eigenfunctions in higher rank}, arXiv:2202.06203 (2022).

\bibitem[GKW15]{GGKW15}
O.~Guichard, F.~Kassel, and A.~Wienhard, \emph{Tameness of {R}iemannian locally
  symmetric spaces arising from {A}nosov representations}, arXiv:1508.04759
  (2015).

\bibitem[GW12]{GW12}
O.~Guichard and A.~Wienhard, \emph{Anosov representations: domains of
  discontinuity and applications}, Invent. Math. \textbf{190} (2012), no.~2,
  357--438.

\bibitem[HC66]{HC66}
Harish-Chandra, \emph{{Discrete series for semisimple Lie groups. II: Explicit
  determination of the characters}}, Acta Math. \textbf{116} (1966), 1--111.

\bibitem[Hel84]{gaga}
S.~Helgason, \emph{Groups and geometric analysis: integral geometry, invariant
  differential operators, and spherical functions}, Pure and applied
  mathematics, Academic Press, 1984.

\bibitem[Hum92]{Hum92}
James~E. Humphreys, \emph{Reflection groups and coxeter groups}, 1992, Hier
  auch später erschienene, unveränd. Nachdrucke.

\bibitem[KL18]{KL18}
M.~Kapovich and B.~Leeb, \emph{Finsler bordifications of symmetric and certain
  locally symmetric spaces}, Geom. Topol. \textbf{22} (2018), no.~5,
  2533--2646.

\bibitem[Kna86]{Kna86}
A.W. Knapp, \emph{Representation theory of semisimple groups: An overview based
  on examples}, Princeton Univ. Press, 1986.

\bibitem[Lab06]{Lab06}
F.~Labourie, \emph{{A}nosov flows, surface groups and curves in projective
  space}, Invent. Math. \textbf{165} (2006), no.~1, 51--114.

\bibitem[LM09]{lapidmueller09}
E.~Lapid and W.~M{\"u}ller, \emph{{Spectral asymptotics for arithmetic
  quotients of $SL(n, \mathbb R)/SO(n)$}}, Duke Math. J. \textbf{149} (2009),
  no.~1, 117--155.

\bibitem[LP82]{laxphil}
P.~Lax and R.S. Phillips, \emph{The asymptotic distribution of lattice points
  in {E}uclidean and non-{E}uclidean spaces}, J. Funct. Anal. \textbf{46}
  (1982), 280–350.

\bibitem[LV07]{LindenVenkatesh}
E.~Lindenstrauss and A.~Venkatesh, \emph{{Existence and Weyl's law for
  spherical cusp forms}}, Geom. Funct. Anal. \textbf{17} (2007), no.~1,
  220--251.

\bibitem[M{\"u}l07]{Mue07}
W.~M{\"u}ller, \emph{Weyl's law for the cuspidal spectrum of {{\(SL(n)\)}}},
  Ann. Math. \textbf{165} (2007), no.~1, 275--333.

\bibitem[Pat75]{Patterson}
S.J. Patterson, \emph{The laplacian operator on a {Riemann} surface}, Compos.
  Math. \textbf{31} (1975), no.~1, 83--107.

\bibitem[SP85]{Sarnak85}
P.~Sarnak and R.S. Phillips, \emph{On cusp forms for co-finite subgroups of
  {$PSL(2, \mathbb R)$}}, Invent. Math. \textbf{80} (1985), 339--364.

\bibitem[vdBS87]{vdBanSchl87}
E.P. van~den Ban and H.~Schlichtkrull, \emph{Asymptotic expansions and boundary
  values of eigenfunctions on {R}iemannian symmetric spaces.}, J. reine angew.
  Math. \textbf{380} (1987), 108--165.

\bibitem[vdBS89]{vdBanSchl89LocalBD}
\bysame, \emph{Local boundary data of eigenfunctions on a {R}iemannian
  symmetric space}, Invent. Math. \textbf{98} (1989), 639--657.

\end{thebibliography}

\bigskip
\bigskip

\end{document}